\documentclass[a4paper,12pt]{amsart}

\usepackage{amssymb, manfnt}
\usepackage{hyperref}
\usepackage[normalem]{ulem}
\usepackage{tikz, epsfig}
\usetikzlibrary{patterns}
\usepackage{enumitem}

\usepackage[all]{xy}
\usepackage{color}
\usepackage{epsfig}
\numberwithin{equation}{section}

\definecolor{darkgreen}{rgb}{0.008,0.417,0.067}
\definecolor{darkred}{rgb}{0.84,0,0}
\def\smallspace{\negthickspace}
\newcommand\mydef{:=}

\newcounter{subequation}[equation]
\renewcommand{\thesubequation}{\theequation\text{\alph{subequation}}}
\let\extralabel=\label
\newcommand{\subeq}[1]
  {\bgroup\refstepcounter{subequation}\extralabel{#1}\egroup(\thesubequation)}

\newcounter{theorem}
\newtheorem{thm}[theorem]{Theorem}
\newtheorem*{thma*}{Theorem~A}
\newtheorem*{thmb*}{Theorem~B}
\newtheorem*{thmc*}{Theorem~C}
\newtheorem{lemma}[theorem]{Lemma}
\newtheorem{prop}[theorem]{Proposition}
\newtheorem{cor}[theorem]{Corollary}

\theoremstyle{remark}
\newtheorem*{remark*}{Remark}
\newtheorem{remark}[theorem]{Remark}

\newtheorem{example}[theorem]{Example}

\theoremstyle{definition}
\newtheorem{defn}[theorem]{Definition}
\newtheorem{ntn}[theorem]{Notation}

\def\smallspace{\negthickspace}

\newcommand{\espan}{\operatorname{span}}

\newcommand{\stab}[1]{{\operatorname{Stab}_G(#1)}}

\newcommand{\supp}{\operatorname{supp}}

\newcommand{\algtensor}{\otimes}

\newcommand{\CC}{\mathbb{C}}
\newcommand{\NN}{\mathbb{N}}

\newcommand{\ZZ}{\mathbb{Z}}

\newcommand{\FF}{\mathbb{F}}

\newcommand{\Bb}{\mathcal{B}}
\newcommand{\Dd}{\mathcal{D}}
\newcommand{\Ee}{\mathcal{E}}
\newcommand{\Ff}{\mathcal{F}}
\newcommand{\Ww}{\mathcal{W}}
\newcommand{\Gg}{\mathcal{G}}

\newcommand{\Ll}{{L}}

\newcommand{\Oo}{\mathcal{O}}
\newcommand{\Ss}{\mathcal{S}}

\newcommand{\Mm}{\mathcal{M}}
\newcommand{\Jj}{\mathcal{J}}
\newcommand{\act}{\cdot}

\newcommand{\minimatrix}[4]
{\bigl(\begin{smallmatrix}
#1 & #2 \\ #3 & #4\end{smallmatrix} \bigr)}

\date{\today}

\author[R.~Hazrat]{Roozbeh Hazrat}
\address{(Hazrat)
School of Computing, Engineering and Mathematics\\
	Western Sydney University\\
	Parramatta NSW 2150\\
	AUSTRALIA}
\email{R.Hazrat@westernsydney.edu.au}

\author[D.~Pask]{David Pask}
\address{(Pask, Sierakowski, Sims)
School of Mathematics and Applied Statistics \\
	University of Wollongong\\
	Wollongong NSW 2522\\
	AUSTRALIA}	
\email{dpask, asierako, asims@uow.edu.au}
\author[A.~Sierakowski]{Adam Sierakowski}
\author[A.~Sims]{Aidan Sims}

\thanks{This research was supported by the Australian Research Council.}

\title[An algebraic analogue of Exel--Pardo $C^*$-algebras]{An algebraic analogue of Exel--Pardo $C^*$-algebras}

\date{\today}

\newlength{\hatchspread}
\newlength{\hatchthickness}
\newlength{\hatchshift}
\newcommand{\hatchcolor}{}
\tikzset{hatchspread/.code={\setlength{\hatchspread}{#1}},
         hatchthickness/.code={\setlength{\hatchthickness}{#1}},
         hatchshift/.code={\setlength{\hatchshift}{#1}},
         hatchcolor/.code={\renewcommand{\hatchcolor}{#1}}}
\tikzset{hatchspread=3pt,
         hatchthickness=0.4pt,
         hatchshift=0pt,
         hatchcolor=black}
\pgfdeclarepatternformonly[\hatchspread,\hatchthickness,\hatchshift,\hatchcolor]
   {custom north west lines}
   {\pgfqpoint{\dimexpr-2\hatchthickness}{\dimexpr-2\hatchthickness}}
   {\pgfqpoint{\dimexpr\hatchspread+2\hatchthickness}{\dimexpr\hatchspread+2\hatchthickness}}
   {\pgfqpoint{\dimexpr\hatchspread}{\dimexpr\hatchspread}}
   {
    \pgfsetlinewidth{\hatchthickness}
    \pgfpathmoveto{\pgfqpoint{0pt}{\dimexpr\hatchspread+\hatchshift}}
    \pgfpathlineto{\pgfqpoint{\dimexpr\hatchspread+0.15pt+\hatchshift}{-0.15pt}}
    \ifdim \hatchshift > 0pt
      \pgfpathmoveto{\pgfqpoint{0pt}{\hatchshift}}
      \pgfpathlineto{\pgfqpoint{\dimexpr0.15pt+\hatchshift}{-0.15pt}}
    \fi
    \pgfsetstrokecolor{\hatchcolor}
    \pgfusepath{stroke}
   }

\begin{document}

\begin{abstract}
We introduce an algebraic version of the Katsura $C^*$-algebra of a pair $A,B$ of integer matrices and an algebraic version of the Exel--Pardo $C^*$-algebra of a self-similar action on a graph. We prove a Graded Uniqueness Theorem for such algebras and construct a homomorphism of the latter into a Steinberg algebra that, under mild conditions, is an isomorphism. Working with Steinberg algebras over non-Hausdorff groupoids we prove that in the unital case, our algebraic version of Katsura $C^*$-algebras are all isomorphic to Steinberg algebras.
\end{abstract}

\maketitle
\section*{Introduction}

Recently in \cite{MR3581326} Exel and Pardo introduced $C^*$-algebras $\Oo_{G,E}$ giving a unified
treatment of two classes of $C^*$-algebras which have attracted significant recent attention,
namely Katsura $C^*$-algebras \cite{MR2400990} and Nekrashevych's self-similar group
$C^*$-algebras \cite{NekraJO,NC}. Katsura $C^*$-algebras are important as they provide concrete
models for all UCT Kirchberg algebras \cite{Kat2}; self-similar $C^*$-algebras have provided the
first known example of a groupoid whose $C^*$-algebra is simple and whose Steinberg algebra over
some fields is non-simple \cite{CESS}. This suggests that algebraic analogues of Exel--Pardo
$C^*$-algebras may be a source of interesting new examples. Partial results have already been
established by Clark, Exel and Pardo in \cite{ClaExePar} who introduced
$\mathcal{O}_{(G,E)}^{\text{alg}}(R)$ for finite graphs $E$. However, as far as we know this
article is the first study of algebraic analogues of Exel--Pardo $C^*$-algebras for infinite
graphs.

In this paper we introduce an algebraic version of Exel--Pardo $C^*$-algebras $\Oo_{G,E}$
providing novel results both when $E$ is infinite and when it is finite. We focus on graphs $E$
that are row-finite with no sources. Up to Morita equivalence, they include all Exel--Pardo
$C^*$-algebras $\Oo_{G,E}$ \cite[Theorem~3.2]{ExeParSta}.

The paper is structured as follows. In Section~\ref{section.one} we introduce the $^*$-algebras
$\Ll_R(G,E)$ as an algebraic version of Exel--Pardo $C^*$-algebras. The main ingredient is an
action of a group $G$ on a graph $E$ which incorporates a notion of a remainder. More
specifically, we take a countable discrete group $G$, an action
\begin{align*}
&(g,v)\mapsto g\act{}v, \ \ \ (g,e)\mapsto g\act{}e
\end{align*}
of $G$ on a row-finite graph $E = (E^0, E^1, r , s )$ with no sources, and a one-cocycle
$\varphi:G\times E^1 \to G$
for the action of $G$ on the edges of $E$. Following \cite{MR3581326}, with this data and a few natural axioms we get an action of $G$ on the space of finite paths $E^*$ which satisfies the following ``self-similarity'' equation
\begin{equation}
\label{selfsim}
g\act{}(\alpha\beta)=(g\act{}\alpha)(\varphi(g,\alpha)\act{}\beta), \ \ \ \text{for all} \ \ g\in G, \alpha\beta\in E^*.
\end{equation}
We refer to Notation~\ref{note} for careful exposition of this setup.

Our algebraic analogue of the Exel--Pardo algebras is described by generators and relations.
Specifically, given a self-similar action of $G$ on $E$ as above, we consider $^*$-algebras
generated by elements $p_{v,h}$ and $s_{e,g}$ indexed by vertices $v$ and edges $e$ of $E$, and by
elements $g,h$ of $G$ under relations that ensure that: the $p_{v, e_G}$ and $s_{e, e_G}$ form an
$E$-family in the sense of Leavitt-path algebras; and for each $v$, the map $g \mapsto p_{v,g}$ is a
unitary representation of $G$; and multiplication amongst the $p_{v,h}$ and $s_{e,g}$ reflect the
structure of the self-similar action (see Definition~\ref{def.LGE}). We prove in
Theorem~\ref{thm.univarsal} that up to $^*$-isomorphism there exists a unique $^*$-algebra
$\Ll_R(G,E)$ over a commutative unital ring $R$ universal for these generators and relations.

The analogous $C^*$-algebras $\Oo_{G, E}$ were first studied for finite graphs in \cite{MR3581326}
and then for countably infinite graphs in \cite{ExeParSta}; our generators and relations are
modelled on the latter, and determine the same $^*$-algebra over $R$, though we omit the proof of
this assertion. For finite graphs, an algebraic analogue $\mathcal{O}_{(G,E)}^{\text{alg}}(R)$ has
also been studied in \cite{ClaExePar}. Our generators and relations appear different to those in
\cite{ClaExePar} because we include a representation $g \mapsto p_{v,g}$ of $G$ for each $v
\in E^0$, rather than a single representation $g \mapsto u_g$ of $G$. This is to avoid the
use of multiplier rings, which would otherwise be necessary in the setting of graphs with
infinitely many vertices; but we show in Proposition~\ref{prop.unital} that our construction
coincides with that of \cite{ClaExePar} when $E^0$ is finite.

After Proposition~\ref{prop.unital}, we consider the algebraic analogue of Katsura $C^*$-algebras which we denote $\Oo_{A,B}^\text{alg}(R)$ and also revisit the algebraic analogue of graph algebras, the Leavitt path algebras $\Ll_R(E)$ \cite{MR2172342}. We prove in Proposition~\ref{prop.AB} and Proposition~\ref{triv.group.case} that both of these are special cases of Exel--Pardo $^*$-algebras. Beyond $\Oo_{(G,E)}^\text{alg}(R)$, $\Oo_{A,B}^\text{alg}(R)$, $\Ll_R(E)$, there are other examples of Exel--Pardo $^*$-algebras, but these examples give a good indication of the scope of the class of algebras we consider. We refer to Figure~\ref{fbuha} for a schematic comparison of the examples in Section~\ref{section.one}.

In Section~\ref{section.two} we study the graded structure of $\Ll_R(G,E)$. In Lemma~\ref{lem.grading} we prove that $\Ll_R(G,E)$ admits a $\ZZ$-grading and we use it to prove the Graded Uniqueness Theorem~\ref{thma}{A}:

\begin{thma*}[Graded Uniqueness]
\label{thma}
Let $(G,E,\varphi)$ be as in Notation~\ref{note}. Let $R$ be a unital commutative $^*$-ring.
Let $\pi \colon \Ll_R(G,E) \to B$ be a $\ZZ$-graded $^*$-algebra homomorphism into a $\ZZ$-graded $^*$-algebra $B$. Suppose that
$$\pi(a)\neq 0 \ \ \ \text{for all} \ \ \ a\in \espan_R\{p_{v,f}: v\in E^0, f\in G\}\setminus \{0\},$$
then $\pi$ is injective.
\end{thma*}
We then study the subalgebra $\Dd\mydef\espan_R\{p_{v,f}: v\in E^0, f\in G\}$ inside of $\Ll_R(G,E)$. We provide a structural characterisation of $\Dd$ as a direct sum of matrix algebras of certain $^*$-algebras over $R$ (see Theorem~\ref{thmb.new}).

In Section~\ref{section.three} we revisit and generalise some of the work of Exel, Pardo and Clark in \cite{ClaExePar}. There they considered (among other things) an algebraic analogue of the well known isomorphism of \cite{MR3581326} between $\Oo_{G,E}$ and the groupoid  $C^*$-algebra $C^*(\Gg_{\text{tight}}(\Ss_{G,E}))$ associated to a groupoid of germs constructed from $G$ and $E$ (see Definition~\ref{def.GSEG}). In particular, they proved that $\Oo_{(G,E)}^\text{alg}(R)$ is isomorphic to the Steinberg algebra $A_R(\Gg_{\text{tight}}(\Ss_{G,E}))$ whenever $E$ is finite and $R=\CC$ (see Remark~\ref{the.remark}). In Proposition~\ref{nonzero} we prove that there always exists a canonical $^*$-homomorphism $$\pi_{G,E}\colon \Ll_R(G,E) \to A_R(\Gg_{\text{tight}}(\Ss_{G,E})).$$ When the groupoid is Hausdorff we prove that this $\pi_{G,E}$ is an isomorphism:
\begin{thmb*}[The Hausdorff case]
\label{thmb}
Let $(G,E,\varphi)$ be as in Notation~\ref{note}. Let $R$ be a unital commutative $^*$-ring. If $\Gg_{\text{tight}}(\Ss_{G,E})$ is Hausdorff then
$$\pi_{G,E}\colon \Ll_R(G,E) \to A_R(\Gg_{\text{tight}}(\Ss_{G,E}))$$
from Proposition~\ref{nonzero} is a $^*$-isomorphism.
\end{thmb*}
Using Theorem~\ref{thmb}{B} we can apply existing machinery \cite{AdelaR, ExePar, MR3581326} to describe precisely when $\Ll_R(G,E)$ is simple and provide sufficient conditions for $\Ll_R(G,E)$ to be  simple and purely infinite. We do this in Proposition~\ref{simple.LGE} and Proposition~\ref{pi.LGE}.

We have not proved a general non-Hausdorff version of Theorem~\ref{thmb}{B}, but we obtain partial results in Section~\ref{section.four}. In particular we show that the $^*$-algebra $\Oo_{A,B}^\text{alg}(R)$ associated to finite integer matrices $A, B$ is always isomorphic to the associated Steinberg algebra:
\begin{thmc*}[Steinberg--Katsura $^*$-algebras]
\label{thmc}
Fix $N\in \NN$ and matrices $A, B\in M_N(\ZZ)$ such that $A_{ij}\geq 0$ and $\sum_jA_{ij}>0$ for all $i$. Let $\Gg_{A,B}$ be the groupoid of germs for the Katsura triple $(\ZZ,E,\varphi)$ associated to $A$ and $B$ as in Definition~\ref{def.graph.for.AB} and Definition~\ref{def.GSEG}. Then
$$\Ll_R(\ZZ,E) \cong \Oo_{A,B}^\text{alg}(R)\cong A_R(\Gg_{A,B}).$$
\end{thmc*}

Theorem~\ref{thmc}{C} is obtained by proving $\Ll_R(G,E) \cong A_R(\Gg_{\text{tight}}(\Ss_{G,E})) $ for a broad class of self-similar actions $(G,E,\varphi)$  (see Theorem~\ref{non.Hasu.EG}) and then applying this result to Katrura triples.

\numberwithin{theorem}{section}
\section{The algebraic version of Exel--Pardo $C^*$-algebras}
\label{section.one}

The $C^*$-algebras $\Oo_{G,E}$ unify many previously known classes of $C^*$-algebras, including graph $C^*$-algebras, Katsura $C^*$-algebras and $C^*$-algebras associated to self-similar groups \cite{MR3581326}. In the algebraic setting much less is known. In this section we define the Exel--Pardo $^*$-algebra $\Ll_R(G,E)$ as an algebraic analogue of the Exel--Pardo $C^*$-algebra and compare it to other known algebras. Exel, Clark and Pardo have already made a definition of $\Oo_{(G,E)}^\text{alg}(R)$ when $E$ is finite and we show that our definition is a genuine generalisation of theirs (see Proposition~\ref{prop.unital}). We do not attempt to give an exhaustive list of examples, but we will consider how $\Ll_R(G,E)$ relates to an algebraic analogue of Katsura $C^*$-algebras associated to infinite matrices, and an algebraic analogue of graph $C^*$-algebras of infinite graphs, the Levitt path algebras $\Ll_R(E)$.

We start with a few definitions. Following \cite{MR1626528},
a \emph{directed graph} $E$ consists of countable sets $E^0$, $E^1$
of vertices and edges,
and maps $r, s \colon E^1 \to E^0$
describing the range and source of edges.
The graph is \emph{row-finite} if $vE^1\mydef r^{-1}(v)$ is finite for each $v\in E^0$, and has \emph{no sources} if $vE^1$ is non-empty for each $v\in E^0$.

A $^*$-algebra over a $^*$-ring $R$ is an algebra equipped with a map $a\mapsto a^*$ called an \emph{involution} satisfying that $(a^*)^*=a$, $(ab)^*=b^*a^*$ and $(ra+b)^*=r^*a^*+b^*$. Let $A$ be such an algebra. We call $p\in A$ a \emph{projection} if $p=p^*=p^2$, we call $s\in A$ a \emph{partial isometry} if $s=ss^*s$ and we call $u\in A$ a \emph{partial unitary} if $u^*u=uu^*=(u^*u)^2$. Two projections are mutually \emph{orthogonal} if their product is zero.

In this paper we use the convention from \cite{MR2135030} where paths read from right to left when defining graph algebras, hence the adjusted Definition~\ref{def.E.family}.
\begin{defn}[{\cite{MR1626528}}]
\label{def.E.family}
If $E$ is a row-finite directed graph, an \emph{$E$-family} in a $^*$-algebra $A$ consists of a set $\{P_v : v\in E^0\}$ of mutually orthogonal projections and a set $\{S_e : e\in E^1\}$ of partial isometries in $A$ such that
\begin{align}
\label{ref11}&S_{e}^*S_{f}=\delta_{e,f} P_{s(e)} \ \text{for all}\ e,f\in E^1, \ \text{and}\\
\label{ref12}&P_v=\sum_{e\in vE^1} S_{e}S_{e}^* \ \text{for all} \ v\in r(E^1).
\end{align}
\end{defn}

Let $E$ be a directed graph. Following \cite{MR3581326}, by an \emph{automorphism} of $E$ we mean a bijective map $\sigma \colon E^0\,\sqcup\, E^1 \to E^0\,\sqcup\, E^1$ such that $\sigma(E^i ) = E^i$, for $i=0,1$ and such that $r \circ \sigma = \sigma \circ r$, and $s \circ \sigma = \sigma \circ s$. By an \emph{action} of a group $G$ on $E$ we shall mean a group homomorphism $g\mapsto \sigma_g$ from $G$ to the group of all automorphisms of $E$. We often write $g\act{e}$ instead of $\sigma_g(e)$. The unit in a group $G$ is denoted $e_G$.

Let $X$ be a set, and let $\sigma$ be an action of a group $G$ on $X$ (i.e., a homomorphism from $G$ to the group of bijections from $X$ to $X$). A map $\varphi\colon G\times X\to G$ is a \emph{one-cocycle} for $\sigma$ if $$\varphi(gh,x)=\varphi(g, \sigma_h(x))\varphi(h,x)$$ for all $g, h\in G$, and all $x\in X$, see \cite{MR3581326}.

\begin{ntn}\label{note}
The quadruple $(G,E,\sigma,\varphi)$, sometimes written as a triple $(G,E,\varphi)$ or a pair $(G,E)$, will denote a countable discrete group $G$, a row-finite graph $E$ with no sources, an (occasionally unnamed) action $\sigma$ of $G$ on $E$ and a one-cocycle $\varphi \colon {G \times E^1}\to G$ for the restriction of $\sigma$ to $E^1$ such that for all $g\in G, e\in E^1, v\in E^0$
\begin{equation}
\label{axionGE}
\varphi(g,e)\act{}v=g\act{}v.
\end{equation}
\end{ntn}

\begin{remark}
The axiom \eqref{axionGE} implies the apparently more general self-similarity condition \eqref{selfsim} (see Lemma \ref{tech.lemma}\eqref{elfsim}). However,  \eqref{selfsim} can also be obtained if we only assume the condition that $\varphi(g,e)\act{}s(e)={g\act{}s(e)}$ whenever  $g\in G, e\in E^1$. Thus, the constraint \eqref{axionGE} might seem unnatural. However, as shown in \cite{MR3581326}, the most prominent classes of examples satisfy this constraint, see \cite[p.~1049]{MR3581326}. To remove this constraint it is arguably more natural to work in the setting of self-similar action of groupoids as in \cite{LRRW}.
\end{remark}


\begin{remark}
\label{3cases}
In this section we consider 3 types of triples $(G,E,\varphi)$:
\begin{enumerate}
\item the triples $(G,E,\varphi)$ where $E$ is finite.
\item the Katsura triples as defined in Definition~\ref{def.graph.for.AB}.
\item the triples $(G,E,\varphi)$ where $G$ is trivial.
\end{enumerate}
In subsection  \ref{The unital case}--\ref{triv.group} we will see how they generate 3 important classes of algebras that we schematically illustrate on Figure~\ref{fbuha}. Each of these serves as an example of our more general construction.
\end{remark}

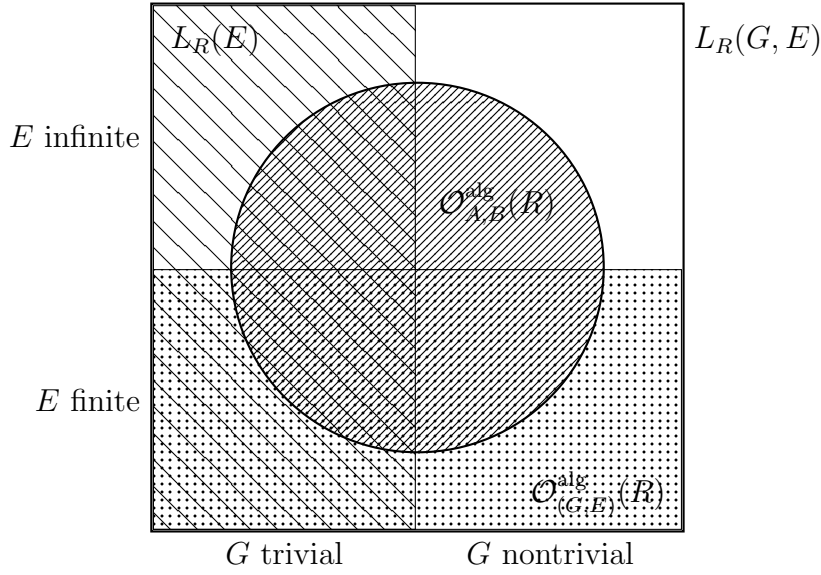
\begin{figure}
\centering
\[
\begin{tikzpicture}[scale=7]
\draw[thick]
  (0,0) rectangle (1,1);
\draw
  (0,0.75) node[anchor=east]  {$E$ infinite}
  (0,0.25) node[anchor=east]  {$E$ finite}
  (0.25,0) node[anchor=north] {$G$ trivial}
  (0.75,0) node[anchor=north] {$G$ nontrivial};
\draw
  (0.12,0.93) node {$L_R(E)$}
    (1.14,0.93) node {$\Ll_R(G,E)$}
  (0.84,0.07) node {$\mathcal{O}_{(G,E)}^{\text{alg}}(R)$}
 (0.65,0.62) node {$\Oo_{A,B}^\text{alg}(R)$};
   \draw [black, thick, pattern=north east lines] (0.5,0.5) circle [radius=0.35];
   \draw [black, pattern=custom north west lines, hatchspread=10pt] (0.004,0.004) rectangle (0.496,0.996);
   \draw [black, pattern=dots] (0.004,0.004) rectangle (0.996,0.496);
\end{tikzpicture}
\]
\caption{The dotted area indicates that each unital Exel--Pardo $^*$-algebra is isomorphim to some $\mathcal{O}_{(G,E)}^{\text{alg}}(R)$.} \label{fbuha}
\end{figure}

We now present a bit more terminology and then state Theorem~\ref{thm.univarsal}, which asserts the existence and uniqueness of the $^*$-algebra $\Ll_R(G,E) $ of the triple $(G,E,\varphi)$.

\begin{defn}\label{def.LGE}
Let $(G,E,\varphi)$ be as in Notation~\ref{note}. Let $R$ be a unital commutative $^*$-ring. A \emph{$(G,E)$-family} in a $^*$-algebra $A$ over $R$ is a set
$$\{P_{v,f} : v\in E^0, f\in G\}\cup \{S_{e,g} : e\in E^1, g\in G\}\subseteq A$$
such that
\begin{enumerate}[label=(\alph*)]
\item\label{GET.a} $\{P_{v,e_G} : v\in E^0\}\cup \{S_{e,e_G} : e\in E^1\}$ is an $E$-family in $A$,
\item\label{GET.b} $(P_{v,f})^*=P_{f^{-1}\act{}v,f^{-1}}$,
\item\label{GET.c} $P_{v,f}P_{w,h}=\delta_{v,f\act{}w}P_{v,fh}$
\item\label{GET.d} $P_{v,f}S_{e,g}=\delta_{v,r(f\act{}e)}S_{f\act{}e,\varphi(f, e)g}$, and
\item\label{GET.e} $S_{e,g}P_{v,f}=\delta_{g\act{}v,s(e)}S_{e,gf}$.
\end{enumerate}
We shall often abbreviate a $(G,E)$-family as $\{P_{v,f}, S_{e,g}\}$.
\end{defn}

\begin{thm}[Universal Property]
\label{thm.univarsal}
Let $(G,E,\varphi)$ be as in Notation~\ref{note}. Let $R$ be a unital commutative $^*$-ring. Then there is a $^*$-algebra $\Ll_R(G,E)$ over $R$ generated by a $(G,E)$-family $\{p_{v,f}, s_{e,g}\}$ with the following universal property: whenever $\{P_{v,f}, S_{e,g}\}$ is a $(G,E)$-family in a $^*$-algebra $A$ over $R$, there is a unique $^*$-homomorphism $\pi_{G,E}\colon \Ll_R(G,E) \to A$ such that
\begin{eqnarray*}
\pi_{G,E}(p_{v,f})=P_{v,f}, \ \ \text{and}\ \ \ \pi_{G,E}(s_{e,g})=S_{e,g}
\end{eqnarray*}
for all $v\in E^0$, $e\in E^1$ and $f,g\in G$.
\end{thm}

We call the $^*$-algebra $\Ll_R(G,E)$ of Theorem~\ref{thm.univarsal} the \emph{Exel--Pardo $^*$-algebra of $(G,E)$}, and we call $\{p_{v,f}, s_{e,g}\}$ the \emph{universal $(G,E)$-family}.

\begin{proof}[Proof of Theorem~\ref{thm.univarsal}]
Let
$$X\mydef \{ P_{v,f}, S_{e,g}, (P_{v,f})^*, (S_{e,g})^*: v\in E^0, e\in E^1, f,g\in G\}$$
be a set of formal symbols and $Y\mydef w(X)$ the set of all finite words in the alphabet $X$. Let $\FF_R(Y)$ be the free $R$-module generated by $Y$, that is $\FF_R(Y)$ is the set of formal sums $\sum_{y\in Y}r_y y$ in which all but finitely many coefficients $r_y\in R$ are zero. We equip $\FF_R(Y)$ with the multiplication
\begin{align*}
\Big(\sum_{x\in Y}r_x x\Big)\Big(\sum_{y\in Y}s_y y)\mydef \sum_{z\in Y}\sum_{\{x,y\in Y : \ xy=z\}}r_xs_yz.
\end{align*}
We define $((P_{v,f})^*)^*=P_{v,f}$ and $((S_{e,g})^*)^*=S_{e,g}$. For $x=x_1x_2\dots x_n\in Y$,  we define $x^*\mydef x_n^*x_{n-1}^*\dots x_1^*$. We then define $^*\colon \FF_R(Y)\to \FF_R(Y)$ by
\begin{align*}
\Big(\sum_{x\in Y}r_x x\Big)^*\mydef\sum_{x\in Y}r_x^* x^*.
\end{align*}
This makes $\FF_R(Y)$ it into a $^*$-algebra over $R$.

Let $I$ be the two-sided ideal of $\FF_R(Y)$ generated be the union of the following nine sets and their set adjoints:
\begin{align}
   \label{relations2}
   \begin{split}
     &\{P_{v,e_G}-P_{v,e_G}^*, \ P_{v,e_G}^2-P_{v,e_G}: v\in E^0\},\\
&\{S_{e,e_G}S_{e,e_G}^*S_{e,e_G}-S_{e,e_G} : e\in E^1\},\\
&\{S_{e,e_G}^*S_{e,e_G}-P_{s(e),e_G} : e\in E^1\},\\
&\{P_{v,e_G} - \sum_{e\in vE^1} S_{e,e_G}S_{e,e_G}^*:  \ v\in r(E^1)\},\\
&\{  (S_{e,e_G}S_{e,e_G}^*)(S_{f,e_G}S_{f,e_G}^*) :  v\in r(E^1),\, e,f\in vE^1, e\neq f \},\\
&\{(P_{v,f})^*-P_{f^{-1}\act{}v,f^{-1}} : v\in E^0, f\in G\},\\
&\{P_{v,f}P_{w,h}-\delta_{v,f\act{}w}P_{v,fh} : v,w\in E^0, f,h\in G\},\\
&\{P_{v,f}S_{e,g}-\delta_{v,r(f\act{}e)}S_{f\act{}e,\varphi(f, e)g}: v\in E^0, e\in E^1, f,g\in G\},\\
&\{S_{e,g}P_{v,f}-\delta_{g\act{}v,s(e)}S_{e,gf}: v\in E^0, e\in E^1, f,g\in G\}.
   \end{split}
\end{align}

We now define $\Ll_R(G,E)\mydef \FF_R(Y)/I$ and let $\{p_{v,f}, s_{e,g}\}$ be the image of $\{P_{v,f}, S_{e,g}\}$ via the  quotient map $q\colon  \FF_R(Y) \to \Ll_R(G,E)$. By construction the collection $\{p_{v,f}, s_{e,g}\}$ is a $(G,E)$-family in $\Ll_R(G,E)$.

Now let $\{\tilde P_{v,f}, \tilde S_{e,g}\}$ be any $(G,E)$-family in a $^*$-algebra $A$ over $R$. Define $\phi\colon X \to A$ by
$\phi(P_{v,f})=\tilde P_{v,f}$,
$\phi(S_{e,g})=\tilde S_{e,g}$,
$\phi((P_{v,f})^*)=(\tilde P_{v,f})^*$, and
$\phi((S_{e,g})^*)=(\tilde S_{e,g})^*$.
Extend $\phi$ to a map also denoted $\phi\colon Y \to A$ via $\phi(x_1x_2\dots x_n)\mydef \phi(x_1)\phi(x_2)\dots \phi(x_n)$ for  $x_i\in X$. Recall, we may identify $Y$ with a subset of $\FF_R(Y)$ via $y\mapsto 1y$. Now, by the universal property of the free $R$-module $\FF_R(Y)$, there exists a unique $R$-module homomorphism $\Phi\colon \FF_R(Y) \to A$ extending $\phi$. Since $\phi$ is compatible with $^*$, $\Phi$ is a $^*$-homomorphism. Since $\Phi(I)=0$, there is a $^*$-homomorphism $pi_{G,E}\colon \Ll_R(G,E) \to A$ such that $\pi_{G,E}(x+I)=\Phi(x)$ for all $x\in \FF_R(Y)$. For each $v\in E^0, f\in G$ we have
\[\pi_{G,E}(p_{v,f})=\pi_{G,E}(P_{v,f}+I)=\Phi(P_{v,f})=\phi(P_{v,f})=\tilde P_{v,f}.\]
Similarly, for each $e\in E^1,g\in G$, we have $\pi_{G,E}(s_{e,g})=\tilde S_{e,g}$.
\end{proof}

\begin{remark}
\label{rem.nonzero}
We note $\Ll_R(G,E)$ satisfies the following:
\begin{enumerate}
\item When $E^0$ is finite, $\Ll_R(G,E)$ is unital with unit $\sum_{v\in E^0}p_{v,e_G}$.
\item We will show in Proposition~\ref{nonzero} that the generators $p_{v,f}$, $s_{e,g}$ of $\Ll_R(G,E)$ are all nonzero.
\end{enumerate}
\end{remark}

In the following subsection we consider the triples $(G,E,\varphi)$ of Remark~\ref{3cases} and their associated algebras as illustrated on Figure~\ref{fbuha}.

\subsection{The unital case}
\label{The unital case}
For finite graphs $E$, algebraic versions of the Exel--Pardo $C^*$-algebras $\Oo_{G,E}$ were introduced in \cite{ClaExePar}. Here we show that our definition yields the same algebras as those defined in  \cite{ClaExePar}.

First we recall some definitions. Let $A$ be a unital $^*$-algebra. A \emph{unitary representation} of a discrete group $G$ on $A$ corresponds to a collection $\{u^g : g\in G\}$ of unitaries in $A$ satisfying $u^gu^h=u^{gh}$ (for all $g,h\in G$). It follows that $u^{e_G}=1$ and $(u^g)^*=u^{(g^{-1})}$.

Let $(G,E,\varphi)$ be as in Notation~\ref{note} and suppose that $E$ is finite. Let $R$ be a unital commutative ring. In \cite[Definition~6.2]{ClaExePar}, $\Oo_{(G,E)}^\text{alg}(R)$ is defined to be the universal $^*$-algebra over $R$ with the following generators and relations:
\begin{enumerate}
\item Generators:
$\{p_v : v\in E^0\}\cup\{s_e : e\in E^1\} \cup \{u^g : g\in G\}.$
\item Relations:
\begin{enumerate}
\item \label{first} $\{p_v : v\in E^0\}\cup\{s_e : e\in E^1\}$ is an $E$-family.
\item The map $u:G\rightarrow \Oo_{(G,E)}^\text{alg}(R)$ defined by the rule $g\mapsto u^g$ is a unitary representation of $G$.
\item $u^gs_e=s_{g\act{}e}u^{\varphi(g,e)}$ for every $g\in G, e\in E^1$.
\item $u^gp_v=p_{g\act{}v}u^g$ for every $g\in G, v\in E^0$.
\item \label{last} $\sum_{v\in E^0}p_{v}=u^{e_G}$.
\end{enumerate}
\end{enumerate}

\begin{remark}
\label{Cla.remark}
The last relation in the definition of $\Oo_{(G,E)}^\text{alg}(R)$ does not explicitly appear in \cite{ClaExePar}, but was certainly intended. We need this property at the end of the proof of Proposition~\ref{prop.unital}.
\end{remark}

\begin{prop}\label{prop.unital}
Let $(G,E,\varphi)$ be as in Notation~\ref{note} and suppose that $E$ is finite. Equip $R$ with the trivial involution. Then
 $$\Oo_{(G,E)}^\text{alg}(R)\cong \Ll_R(G,E).$$
\end{prop}

\begin{proof}
We first build a homomorphism $\pi_1\colon \Ll_R(G,E)\to \Oo_{(G,E)}^\text{alg}(R)$ by invoking the universal property of $\Ll_R(G,E)$. Let $\{p_v, s_e, u^g\}$ denote the generators for $\Oo_{(G,E)}^\text{alg}(R)$. For each $v\in E^0$, $e\in E^1$ and $f,g\in G$ define
$$P_{v,f}\mydef p_vu^f \quad\text{and}\quad S_{e,g}\mydef s_eu^g.$$
Routine calculations using that $s_e=p_{r(e)}s_{e}p_{s(e)}$ show that $\{P_{v,f}, S_{e,g}\}$ is a $(G,E)$-family.
The universal property of $\Ll_R(G,E)$ now yields a $^*$-algebra homomorphism $$\pi_1\colon \Ll_R(G,E)\to \Oo_{(G,E)}^\text{alg}(R)$$ such that $\pi_{1}(p_{v,f})=P_{v,f}$ and $\pi_{1}(s_{e,g})=S_{e,g}$.

To construct an inverse for $\pi_1$ we will use the universal property of $\Oo_{(G,E)}^\text{alg}(R)$. Using the generators $\{p_{v,f}, s_{e,g}\}$ for $\Ll_R(G,E)$, for each $v\in E^0$, $e\in E^1$ and $g\in G$ define
$$P_{v}\mydef p_{v,e_G}, \ \ \, S_{e}\mydef s_{e,e_G}, \ \ \ U^g\mydef \sum_{v\in E^0} p_{v,g}.$$
Again, routine calculations using that $e_G\act{}v=v$ show that $\{P_v, S_e, U^g\}$ satisfy the relations \eqref{first}--\eqref{last} in the definition of $\Oo_{(G,E)}^\text{alg}(R)$.
The universal property of $\Oo_{(G,E)}^\text{alg}(R)$ provides a $^*$-algebra homomorphism $$\pi_2\colon \Oo_{(G,E)}^\text{alg}(R)\to  \Ll_R(G,E)$$ such that $\pi_{2}(p_{v})=P_{v}$, $\pi_{2}(s_{e})=S_{e}$ and $\pi_{2}(u^g)=U^g$. By computing $\pi_i\circ \pi_j$ ($i\neq j$) on generators we see that $\pi_1$ is an inverse for $\pi_2$; for the generator $u^g$ in $\Oo_{(G,E)}^\text{alg}(R)$ we use \eqref{last} from the definition of $\Oo_{(G,E)}^\text{alg}(R)$ to see that
$$\pi_1\circ \pi_2(u^g)=\sum_{v\in E^0} p_vu^g=u^g.\qedhere$$
\end{proof}


\subsection{The vertex-trivial case} In \cite{MR2400990} Katsura introduced $C^*$-algebras $\Oo_{A,B}$ which we call Katsura $C^*$-algebras. Here we consider an algebraic analogue, denoted $\Oo_{A,B}^\text{alg}(R)$, and prove that all such $^*$-algebras are Exel--Pardo $^*$-algebras using the translation of the matrices $A, B$ into an action of $\ZZ$ on a graph discovered by Exel and Pardo in \cite{ExePar} (see Definition~\ref{def.graph.for.AB}).

We recall the relevant notation needed to introduce Katsura $C^*$-algebras $\Oo_{A,B}$. Fix $N\in \NN\cup \{\infty\}$. Let $ I\mydef\{1,2, \dots, N\}$ for $N$ finite and $I\mydef \NN$ otherwise. With $A_{ij}$ denoting the $ij$-entry of an $I\times I$ nonnegative integer matrix $A$ define $\Omega_A\mydef \{(i,j)\in I\times I: A_{ij}> 0\}$, and $\Omega_A(i)\mydef \{j: (i,j)\in \Omega_A\}$ for $i\in I$.

\begin{defn}[{\cite[Definition~2.3]{ExePar},\cite[Definition~2.2]{MR2400990}}]\label{def.AB}
Fix $N\in \NN\cup \{\infty\}$ and row-finite matrixes $A, B\in M_N(\ZZ)$ such that $A$ has nonnegative entries and no zero rows. Let $ I\mydef\{1,2, \dots, N\}$ for $N$ finite and $I\mydef \NN$ otherwise. The  $C^*$-algebra $\Oo_{A,B}$ is the universal $C^*$-algebra generated by mutually orthogonal projections ${(q_k)}_{k\in I}$, partial unitaries ${(u_k)}_{k\in I}$ with $q_k=u_ku_k^*$, and partial isometries $(s_{ijn})_{(i,j)\in\Omega_A, n\in \ZZ}$ such that
\begin{equation}\label{eq.dag}\tag{\dag}
  \begin{array}{lll}
    \text{(i)} & s_{ijn}u_j=s_{ij(n+A_{ij})}, u_is_{ijn}=s_{ij(n+B_{ij})} \ \text{for} \, (i,j)\in\Omega_A, n\in \ZZ,\\
    \text{(ii)} & s_{ijn}^*s_{ijn}=q_j \ \text{for} \ (i,j)\in\Omega_A, n\in \ZZ,\, \text{and}\\
    \text{(iii)} & q_i=\sum_{j\in \Omega_A(i), 1\leq n \leq A_{ij}} s_{ijn}s_{ijn}^*\, (i\in I).
  \end{array}
\end{equation}
\end{defn}

\begin{defn}\label{rem.def.LAB}
Similarly to the construction of $\Ll_R(G,E)$, it makes perfect sense to consider the universal $^*$-algebra over a unital commutative $^*$-ring $R$ with the same generators and relations as those for $\Oo_{A,B}$ but with the additional relations that if $j\neq j'$ then $s_{ijn}^*s_{ij'n}=0$ for all $n$ and that if $1\leq n<n'\leq A_{ij}$ then $s_{ijn}^*s_{ijn'}=0$. These relations follow automatically from the others in a $C^*$-algebra, but must be imposed separately in an abstract $^*$-algebra. We denote this universal $^*$-algebra by $\Oo_{A,B}^\text{alg}(R)$.
\end{defn}

\begin{defn}[{\cite[Remark 18.3]{MR3581326}}]
\label{def.graph.for.AB}
Fix $N,A, B$ and $I$ as in Definition~\ref{def.AB}.
Let $E$ be the directed graph with vertices $\{v_i: i\in I\}$ and
edges $\{e_{ijn}: i,j\in I, 0\leq n \leq A_{ij}-1\}$ with $r(e_{ijn})=v_i$ and $s(e_{ijn})=v_j$.
By construction $E$ is row-finite and has no sources.
%
%

It is straightforward to check that we can define an action $\sigma$ of $\ZZ$ on $E^1$ and a one-cocycle $\varphi \colon {\ZZ \times E^1}\to \ZZ$ as follows: For any $i,j\in \Omega_A$, $n\in \{0,\dots, A_{ij}-1\}$ and $m\in \ZZ$ there are a unique $\hat n \in \{0,\dots, A_{ij}-1\}$ and a unique $\hat k\in \ZZ$ such that $mB_{ij}+n=\hat k A_{ij} + \hat n$. We define $\sigma_m(e_{ijn}) \mydef e_{ij\hat n}$ and $\varphi(m,e_{ijn})\mydef \hat k$. Since $\sigma_m$ permutes parallel edges, $\sigma$ extends to an action on $E$ such that $\sigma_m(v)=v$ for all $m\in\ZZ$ and $v\in E^0$. We call  $(\ZZ,E,\varphi)$ the \emph{Katsura triple} associated to $A,B$.
\end{defn}


\begin{prop}\label{prop.AB}
Take $N\in \NN\cup \{\infty\}$, and let $A, B\in M_N(\ZZ)$ be as in Definition~\ref{def.AB}. Let $(\ZZ,E,\varphi)$ be the Katsura triple associated to $A,B$. Then
 $$ \Oo_{A,B}^\text{alg}(R)\cong \Ll_R(\ZZ,E),$$
 as $^*$-algebras over $R$.
\end{prop}

\begin{proof}
We first use the universal property of $\Ll_R(\ZZ,E)$ to obtain a homomorphism $\pi_1\colon \Ll_R(\ZZ,E)\to \Oo_{A,B}^\text{alg}(R)$. For this let $\{q_k, u_k, s_{ijn} \}$ denote the generators for $\Oo_{A,B}^\text{alg}(R)$. Let $ I\mydef\{1,2, \dots, N\}$ for $N$ finite and $I\mydef \NN$ otherwise. For each $k\in I$, and each $m\in \ZZ$, set
\[
  u_k^m \mydef
  \begin{cases}
     (u_k^*)^{-m} & \text{if $m<0$,} \\
    q_k & \text{if $m=0$,}\\
        (u_k)^{m} & \text{if $m>0$.}
  \end{cases}
\]
For $v=v_k\in E^0$, $e=e_{ijn}\in E^1$ and $m,l\in \ZZ$ define
$$P_{v,m}\mydef u_k^m, \ \ \, S_{e,l}\mydef s_{ijn}u_j^l.$$
Routine calculations show that $\{P_{v,m}, S_{e,l}\}$ is a $(\ZZ,E)$-family.
%
So the universal property of $\Ll_R(\ZZ,E)$ provides a $^*$-algebra homomorphism $$\pi_1\colon \Ll_R(\ZZ,E)\to \Oo_{A,B}^\text{alg}(R)$$ such that $\pi_{1}(p_{v,m})=P_{v,m}$ and $\pi_{1}(s_{e,l})=S_{e,l}$.

To construct an inverse for $\pi_1$, let $\{p_{v,m}, s_{e,l}\}$ denote the generators for $\Ll_R(\ZZ,E)$. For each $(i,j)\in\Omega_A$ and $m\in \ZZ$, there exists unique elements $n\in \{0, \dots,  A_{ij}-1\}$ and $k\in \ZZ$ such that $m=n+kA_{ij}$. Define
$$Q_{k}\mydef p_{v_k,0}, \ \ \, S_{ijm}\mydef s_{e_{ijn},k}, \ \ \ U_k\mydef p_{v_k,1}.$$

It is routine to see that the $Q_k$ are mutually orthogonal projections, the $S_{ijm}$ are partial isometries and $U_k^*U_k=Q_k=U_kU_k^*$ for each $k\in I$. It remains to show that $\{Q_{k}, S_{ijm}, U_k\}$ satisfy \eqref{eq.dag}.

First we show \eqref{eq.dag}(i).
 Fix $(i,j)\in\Omega_A, m\in \ZZ$. Let $n\in \{0, \dots,  A_{ij}-1\}$ and $k\in \ZZ$ be the  elements such that $m=n+kA_{ij}$. We get the first equality of \eqref{eq.dag}(i) by
$$S_{ij(m+A_{ij})}=S_{ij(n+(k+1)A_{ij})}=s_{e_{ijn},k+1}=s_{e_{ijn},k}p_{v_j,1}=S_{ijm}U_j.$$

For the second equality in \eqref{eq.dag}(i) consider the same $(i,j)$ and $m$. Let $\hat n\in \{0, \dots,  A_{ij}-1\}$ and $\hat k\in \ZZ$ be the elements such that $B_{ij}+n=\hat k A_{ij} + \hat n$. By Definition~\ref{def.graph.for.AB}, $\sigma_1(e_{ijn}) =e_{ij\hat n}$ and $\varphi(1, e_{ijn})= \hat k$. So
\begin{align*}
S_{ij(m+B_{ij})}&=S_{ij(n+kA_{ij}+B_{ij})}=S_{ij((\hat k+k) A_{ij} + \hat n)}=s_{e_{ij\hat n},\hat k+k}.
\end{align*}
Hence
\begin{align*}
U_iS_{ijm}&=p_{v_i,1}s_{e_{ijn},k}=\delta_{v_i,r(1\act{}e_{ijn})}s_{1\act{}e_{ijn},\varphi(1, e_{ijn})+k}=S_{ij(m+B_{ij})}.
\end{align*}

To verify \eqref{eq.dag}(ii) let $m=n+kA_{ij}$ as above.  Then
$$S_{ijm}S_{ijm}^*=s_{e_{ijn},0}p_{v_j,k}p_{v_j,k}^*s_{e_{ijn},0}^*=s_{e_{ijn},0}s_{e_{ijn},0}^*=p_{s(e_{ijn}),0}=Q_j,$$
as required. Finally  \eqref{eq.dag}(iii) follows from the fact that  $\{p_{v,0}, s_{e,0}\}$ is an $E$-family in $\Ll_R(\ZZ,E)$.

The universal property of $\Oo_{A,B}^\text{alg}(R)$ provides a $^*$-algebra homomorphism $$\pi_2\colon \Oo_{A,B}^\text{alg}(R) \to  \Ll_R(\ZZ,E)$$ such that $\pi_{2}(q_k)=Q_k$, $\pi_{2}(u_k)=U_k$ and $\pi_{2}(s_{ijm})=S_{ijm}$. Direct computation on generators shows that  $\pi_1$ and $\pi_2$ are mutually inverse.
\end{proof}

\subsection{The trivial group case}
\label{triv.group}
As our  final example for this section we consider the case where the group $G=\{0\}$. When $G=\{0\}$ the $C^*$-algebra $\Oo_{G,E}$ is isomorphic to the graph $C^*$-algebra $C^*(E)$, see \cite{MR3581326, ExeParSta}. In the algebraic setting we show that for any row-finite graph $E$ with no sources, if $G=\{0\}$ then the Exel--Pardo $^*$-algebra $\Ll_R(G,E)$ is isomorphic to the Leavitt path algebra of $E$.

We start by introducing Leavitt path $R$-algebras, although we reverse the usual edge-direction convention to match the rest of the paper. Let $R$ be a unital commutative ring. Let $E$ be a row-finite graph. As in \cite[p.~161]{MR2310414}, the \emph{Leavitt path $R$-algebra} $L_R(E)$ of $E$ with coefficients in $R$ is the $R$-algebra generated by elements $\{p_v, x_e, y_e : v\in E^0, e\in E^1\}$ such that
\begin{align}
\begin{split}
&p_vp_{v'} = \delta_{v,v'}p_v \ \text{for all}\ v,v' \in E^0,\\
&p_{r(e)}x_e = x_ep_{s(e)} = x_e \ \text{for all}\ e\in E^1,\\
&p_{s(e)}y_e = y_ep_{r(e)} = y_e \ \text{for all}\ e\in E^1,\\
&y_ex_{e'} = \delta_{e,e'}p_{s(e)} \ \text{for all}\ e,e'\in E^1, \ \text{and}\\
&p_v =\sum_{\{e: r(e)=v\}}  x_ey_e \ \text{for all} \ v\in r(E^1).
\label{LER.relations}
\end{split}
\end{align}

As pointed out in \cite[p.~70]{MR3319981} this definition coincides with the one by Abrams and Aranda Pino in \cite[Definition~1.3]{MR2172342}.
Below we show that every Leavitt path algebra, regarded as a $^*$-algebra under the involution such that $(rx_e)^*=r^*y_e$, is an Exel--Pardo $^*$-algebra. For this we firstly confirm that the mentioned property defines an involution on $L_R(E)$.

\begin{lemma}
\label{into.star.alg}
Let $E$ be a row-finite graph with no sources and $R$ unital commutative $^*$-ring. Then there is a unique involution on $L_R(E)$ such that $(r x_e)^*=r^*y_e$ for all $r\in R$ and $e\in E^0$.
\end{lemma}
\begin{proof}
Note that $(rp_v)^*=r^*p_v$ using the last equality of \eqref{LER.relations} and the fact that $E$ is row-finite with no sources.
\end{proof}

For the following proposition  we note that for the self-similar actions $(G,E, \sigma, \varphi)$ considered in this paper, if $G=\{0\}$ then necessarily $\varphi=0$ and $\sigma=\text{id}_E$.

\begin{prop}\label{triv.group.case}
Let $E$ be a row-finite graph with no sources, and consider the quadruple $(\{0\},E, \text{id}_E, 0)$ as in Notation~\ref{note}. Let $R$ be a unital commutative ring. Then there is an $R$-algebra isomorphism $\pi_1\colon L_R(E)\to  \Ll_R(\{0\},E)$ such that
 $$\pi_{1}(p_{v})=p_{v,0}, \ \ \ \pi_{1}(x_{e})=s_{e,0}, \text{and} \ \ \ \pi_{1}(y_{e})=s_{e,0}^*.$$
\end{prop}
\begin{proof}
The defining relations for $\Ll_R(\{0\},E)$ are
\begin{enumerate}
\item[(a)]$\{p_{v,0} : v\in E^0\}\cup \{s_{e,0} : e\in E^1\}$ is an $E$-family in $\Ll_R(\{0\},E)$,
\item[(b)] $(p_{v,0})^*=p_{v,0}$,
\item[(c)]$p_{v,0}p_{w,0}=\delta_{v,w}p_{v,0}$,
\item[(d)]$p_{v,0}s_{e,0}=\delta_{v,r(e)}s_{e,0}$, and
\item[(e)]$s_{e,0}p_{v,0}=\delta_{v,s(e)}s_{e,0}$.
\end{enumerate}
Note that $\{p_{v,0}, s_{e,0}, s_{e,0}^*\}$ satisfy all the relations satisfied by $\{p_v, x_e, y_e\}$ in $L_R(E)$, see \eqref{LER.relations}. Therefore, the universal property of $L_R(E)$ provides an $R$-algebra homomorphism $$\pi_1\colon L_R(E)\to \Ll_R(\{0\},E)$$ such that $\pi_{1}(p_{v})=p_{v,0}$, $\pi_{1}(x_{e})=s_{e,0}$ and $\pi_{1}(y_{e})=s_{e,0}^*$.

We now construct a map in the opposite direction. For this we need all the elements in $L_R(E)$ to have an adjoint. With the trivial adjoint on $R$ we turn $L_R(E)$ into a $^*$-algebra using the adjoint of Lemma~\ref{into.star.alg}. Then $\{p_v, x_e, y_e\}$ satisfy the relations (a)--(e) with $p_{v,0}, s_{e,0}, s_{e,0}^*$ replaced by $p_v, x_e, y_e$. The universal property of $\Ll_R(\{0\},E)$ provides a $^*$-algebra homomorphism
$$\pi_2\colon \Ll_R(\{0\},E)\to L_R(E)$$ such that $\pi_{2}(p_{v,0})=p_{v}$, $\pi_{2}(s_{e,0})=x_{e}$ and $\pi_{2}(s_{e,0}^*)=y_{e}$. We deduce that $\pi_1$ is an $R$-isomorphism with inverse $\pi_2$.
\end{proof}

\section{Proof of Theorem~\ref{thma}{A}}
\label{section.two}
In this section we prove the Graded Uniqueness Theorem~\ref{thma}{A} and the structure result Theorem~\ref{thmb.new}. Much of the work here is inspired by Tomforde who proved the Graded Uniqueness Theorem  \cite[Theorem~5.3]{MR2738365} for Leavitt path algebras. Tomforde proved that a graded homomorphism out of a Leavitt path algebra is injective if it is injective on $\espan_R\{p_{v}: v\in E^0\}$. In our Theorem~\ref{thma}{A} we need the graded homomorphism $\pi$ to be injective on $\Dd\mydef\espan_R\{p_{v,g}: v\in E^0, g\in G\}$; that is, for each $v$ we need to insist that $\pi$ is injective on the image of the group ring $RG$ under the representation $g\mapsto p_{v,g}$.

Theorem~\ref{thmb.new} characterises $\Dd$ as a direct sum of matrix algebras over certain $R$-algebras $\Ww_v$ defined for each $v\in E^0$.  We show that each such $R$-algebra is generated by unitaries $\{W_v^g: g\act{}v=v\}$ inside of a corner of $\Dd$. When looking into how these unitaries behave it turns out that the possibilities are virtually endless. For example, even when $G=\ZZ$ there are cases where all the generators $\{W_v^g: g\act{}v=v\}$ are pairwise distinct and other cases where they all coincide. This has important implications in terms of applying Theorem~\ref{thma}{A} to decide if $\pi$ is injective on $\Dd$, we must first determine the amount of ``collapsing'' that takes place in the canonical homomorphism $g\mapsto p_{v,g}$ and this will vary from example to example.

We now introduce the notation needed to prove Theorem~\ref{thma}{A}. Let $G$ be a discrete group. Following \cite{MR3781435}, a ring $A$ (possibly without unit) is \emph{$G$-graded} if as an additive group it can be written as $A=\bigoplus_{g\in G} A_g $, such that each $A_fA_g\subseteq A_{fg}$. The group $A_g$ is called the \emph{$g$-homogeneous component} of $A$. If $A$ is an algebra over a ring $R$, then $A$ is \emph{$G$-graded} if $A$ is a $G$-graded ring and each $A_g$ is a $R$-submodule of $A$ (i.e., $A_g$ satisfies $RA_g\subseteq A_g$). The elements of $\bigcup_{g\in G} A_g$ in a $G$-graded ring $A$ are called \emph{homogeneous elements} of $A$. The nonzero elements of $A_g$ are called \emph{homogeneous of degree $g$} and we write $\deg(a) = g$ for $a \in A_g\setminus\{0\}$. If $\pi: A\to B$ is a homomorphism between two $G$-graded algebras over a ring $R$, then $\pi$ is a \emph{$G$-graded homomorphism} if $\pi(A_g)\subseteq B_g$ for all $g\in G$.

Let $E$ be a directed graph. We declare vertices to be paths of length $0$ with $r(v)=v=s(v)$. By a path $\alpha$ in $E$ of length $|\alpha|=n \geq  1$, as in \cite[Part 2.3]{MR3581326}, we shall mean any finite sequence of the form $\alpha=\alpha_1\alpha_2\dots\alpha_n$ such that $\alpha_i\in E^1$ and $s(\alpha_i)=r(\alpha_{i+1})$ for all $i$ (this convention agrees with \cite{MR3581326} rather than, for example, \cite{MR3319981}). Here $s(\alpha)\mydef s(\alpha_n)$ and $r(\alpha)\mydef r(\alpha_1)$. For $n\geq 0$ we let $E^n$ denote the set of all paths of length $n$. We let $E^*$ denote the set of all finite paths, so $E^*\mydef\bigcup_{m\geq 0}E^m$. If $\alpha,\beta \in E^*$ satisfy $s(\alpha)=r(\beta)$ we let $\alpha\beta$ be their concatenation.

\begin{defn}
Let $(G,E,\varphi)$ be as in Notation~\ref{note}. Let $R$ be a unital commutative $^*$-ring. For each $g\in G$ and $\alpha=\alpha_1\dots\alpha_n\in E^*$ we define
\[
  s_{\alpha,g} \mydef
  \begin{cases}
     s_{\alpha_1,e_G}s_{\alpha_2,e_G}\dots s_{\alpha_{n-1},e_G}s_{\alpha_{n},g} & \text{if $n>0$,} \\
     p_{\alpha,g} & \text{if $n=0$.}
  \end{cases}
\]
\end{defn}

\begin{remark}
Let $(G,E,\varphi)$ be as in Notation~\ref{note}. Let $R$ be a unital commutative $^*$-ring. Using the relations of a $(G,E)$-family we have
\begin{equation}
\label{small.1}
s_{\alpha,g}=s_{\alpha,g}p_{g^{-1}.s(\alpha),e_G}, \ \ \ s_{\alpha,g}=p_{r(\alpha),e_G}s_{\alpha,g},
\end{equation}
\begin{equation}
\label{small.2}
s_{\alpha,g}=s_{\alpha,e_G}p_{s(\alpha),g},
\end{equation}
\begin{equation}
\label{small.3}
s_{\alpha,g}s_{\beta,h}^*=s_{\alpha,gh^{-1}}s_{\beta,e_G}^*.
\end{equation}
for each $\alpha, \beta\in E^*$ and $g,h\in G$.
\end{remark}

\begin{lemma}
\label{generators.span}
Let $(G,E,\varphi)$ be as in Notation~\ref{note}. Let $R$ be a unital commutative $^*$-ring. Then
$$\Ll_R(G,E)=\espan_R\{s_{\alpha,g}s_{\beta,e_G}^*: \alpha,\beta\in E^*, g\in G, s(\alpha)=g\act{}s(\beta)\}.$$
\end{lemma}
\begin{proof}
Fix $\alpha, \beta\in E^*$ and $g\in G$.  By \eqref{small.1} we have $$s_{\alpha,g}s_{\beta,e_G}^*=\delta_{s(\alpha),g\act{}s(\beta)}s_{\alpha,g}s_{\beta,e_G}^*,$$ so the requirement $s(\alpha)=g\act{}s(\beta)$ is clear. Let $M$ denote the set
$\espan_R\{s_{\mu,e_G}s_{\nu,e_G}^*: \mu,\nu\in E^*\}$. Fix any $\alpha, \beta\in E^*$ and $g, h\in G$. By \eqref{small.1}--\eqref{small.2} we have
$$s_{\alpha,g}=s_{\alpha,e_G}p_{s(\alpha),g}=s_{\alpha,e_G}p_{s(\alpha),e_G}p_{s(\alpha),g}=s_{\alpha,e_G}s_{s(\alpha),e_G}^*p_{s(\alpha),g}\in Mp_{s(\alpha),g}$$ and similarly $s_{\beta,h}\in Mp_{s(\beta),h}$. Using properties of an $E$-family we have $M^*M\subseteq M$. Hence
$$s_{\alpha,g}^*s_{\beta,h}\in (p_{s(\alpha),g})^*Mp_{s(\beta),h}\subseteq \espan_R\{s_{\mu,g}s_{\nu,h}^*: \mu,\nu\in E^*, g,h\in G\}.$$
The desired equality now follows from \eqref{small.3}.
\end{proof}

\begin{remark}
\label{L.span}
Using \eqref{small.2} and borrowing notation from Section~\ref{section.three} (Definition~\ref{semigroup}), $\Ll_R(G,E)=\espan_R\{s_{\alpha,e_G}p_{s(\alpha),g}s_{\beta,e_G}^*: (\alpha,g,\beta)\in  \Ss_{G,E}\}$.
\end{remark}

\begin{lemma}{(cf.~\cite[Proposition~4.7]{MR2738365})}
\label{lem.grading}
Let $(G,E,\varphi)$ be as in Notation~\ref{note}. Let $R$ be a unital commutative $^*$-ring. For each $n\in \ZZ$ define
$$A_n\mydef \espan_R\{s_{\alpha,g}s_{\beta,e_G}^*: \alpha,\beta\in E^*, g\in G, s(\alpha)=g\act{}s(\beta), |\alpha|-|\beta|=n\}.$$
Then $(A_n)_{n\in \ZZ}$ is a $\ZZ$-grading on $\Ll_R(G,E)$.
\end{lemma}
\begin{proof}
Define the symbols $X\mydef \{ P_{v,f}, S_{e,g}, (P_{v,f})^*, (S_{e,g})^*\}$ and the words $Y\mydef\{x_1\dots x_n : n\geq 1, x_i\in X\}$. Recall that from the proof of Thorem~\ref{thm.univarsal}, $\Ll_R(G,E)$ is the quotient of the free $^*$-algebra $\FF_R(Y)$ by the ideal $I$ generated by the elements of the sets \eqref{relations2} and their adjoints.

The $^*$-algebra $\FF_R(Y)$ has a unique $\ZZ$-grading for which the elements $P_{v,f}$, $S_{e,g}$, $(P_{v,f})^*$, $(S_{e,g})^*$ of $X$ have degrees $0, 1, 0$ and $-1$, respectively. Moreover, each generator of $I$ is homogeneous of degree $0$. It follows that $I$ is a graded ideal, in the sense of \cite[Definition~4.6]{MR2738365}, i.e.,
$$I=\bigoplus_{n\in \ZZ} (I\cap \FF_R(Y)_n).$$
Hence $\Ll_R(G,E)$ admits a natural $\ZZ$-grading such that the quotient map $q\colon  \FF_R(Y) \to \Ll_R(G,E)$ a $\ZZ$-graded homomorphism. We see that $A_n=q(\FF_R(Y)_n)$ for all $n$, so $(A_n)_{n\in \ZZ}$ is a $\ZZ$-grading on $\Ll_R(G,E)$.
\end{proof}

With Lemma~\ref{generators.span} and Lemma~\ref{lem.grading} at our disposal we are in position to prove Theorem~\ref{thma}{A}.

\begin{proof}[Proof of Theorem~\ref{thma}{A}]
To ease notation we define $A\mydef \Ll_R(G,E)$. Suppose that $\pi(a)=0$. We must show that $a=0$. Write $a=\sum_{n\in \ZZ}a_n$ such that $a_n\in A_n$ for each $n\in \ZZ$. Since $\pi$ and $B$ are graded, each $\pi(a_n)=0$. Since $(A_n)^*=A_{-n}$, it suffices to show $a_n=0$ for each $n\geq 0$. Fix such $n$ and for convenience set $d\mydef a_n$. We may write $d$ as a finite sum as follows
$$d=\sum_{i\in F}r_i s_{\alpha_i,g_i}s_{\beta_i,e_G}^*=\sum_{i\in F}r_i s_{\alpha_i,e_G}p_{s(\alpha_i),g_i}s_{\beta_i,e_G}^*$$
with $\alpha_i,\beta_i\in E^*, g_i\in G$ such that $|\alpha_i|-|\beta_i|=n$. For each $i\in F$ set $v_i\mydef g_i^{-1}\act{}s(\alpha_i)$. We have
$$p_{s(\alpha_i),g_i}=p_{s(\alpha_i),g_i}p_{v_i,e_G}
=\sum_{e\in v_iE^1}p_{s(\alpha_i),g_i}s_{e, e_G}s_{e, e_G}^*=\sum_{e\in v_iE^1}s_{g_i\act{}e,\varphi(g_i, e)}s_{e, e_G}^*,$$
so we may assume there are $m_1,m_2\in \NN$ such that $\alpha_i\in E^{m_1}$ and $\beta_i\in E^{m_2}$ for all $i\in F$. For each $j\in F$ set $F(j)\mydef \{i\in F : (\alpha_i,\beta_i)=(\alpha_j,\beta_j)\}$. Since $s_{\alpha,e_G}^*s_{\beta,e_G}=\delta_{\alpha,\beta}p_{s(\alpha),e_G}$ for $\alpha, \beta\in E^*$ such that $|\alpha|=|\beta|$
and since $s_{\alpha_j,g_i}=s_{\alpha_j,e_G}p_{s(\alpha_j),g_i}$
we get
\begin{align}
s_{\alpha_j,e_G}^*ds_{\beta_j,e_G}&=\sum_{i\in F(j)}r_i (s_{\alpha_j,e_G}^*s_{\alpha_j,e_G})p_{s(\alpha_j),g_i}(s_{\beta_j,e_G}^*s_{\beta_j,e_G})\\
&=\sum_{i\in F(j)}r_i (p_{s(\alpha_j),e_G})p_{s(\alpha_j),g_i}(p_{s(\beta_j),e_G})\in \Dd.
\end{align}

Since $\pi(d)=\pi(a_n)=0$, we have $\pi(s_{\alpha_j,e_G}^*ds_{\beta_j,e_G})=0$. By injectivity of $\pi$ on $\Dd$ we have $s_{\alpha_j,e_G}^*ds_{\beta_j,e_G}=0$. Since each $s_{\mu,e_G}$ is a partial isometry and $s_{\alpha_j,g_i}=s_{\alpha_j,e_G}p_{s(\alpha_j),g_i}$ we conclude that
$$s_{\alpha_j,e_G}s_{\alpha_j,e_G}^*ds_{\beta_j,e_G}s_{\beta_j,e_G}^*=\sum_{i\in F(j)}r_i s_{\alpha_i,g_i}s_{\beta_i,e_G}^*=0.$$
As $F$ is a disjoint union of subsets of the form $F(j)$, we deduce $a_n=d=0$ as requested.
\end{proof}

We now turn to the proof of Theorem~\ref{thmb.new} describing the $R$-algebra $\Dd$ used in the statement of the Graded Uniqueness Theorem~\ref{thma}{A}. The proof essentially boils down to identifying the appropriate matrix units and algebraic tensor products  inside of $\Dd$. We start by recalling the notion of ``matrix units''.

\begin{defn}
\label{matrix.units}
Let $X$ be a non-empty set. Write $M_X(R)$, or just $M_X$, for the universal $^*$-algebra over $R$ generated by elements $$\{\eta_{x,y} : x, y \in X\}$$ satisfying $\eta_{x,y}^*=\eta_{y,x}$ and $\eta_{x,y}\eta_{w,z}=\delta_{y,w}\eta_{x,z}$. We call the $\eta_{x,y}$ the \emph{matrix units} for $M_X$.
\end{defn}

\begin{lemma}
\label{matrix.decomposition.pre}
Let $(G,E,\varphi)$ be as in Notation~\ref{note}. Let $R$ be a unital commutative $^*$-ring and let $v\in E^0$ be any vertex.  Let $g_v\mydef e_G$. For each $w\in G\act{}v\setminus\{v\}$, fix $g_w\in G$ such that $w=g_w\act{}v$.
\begin{enumerate}
\item \label{item.matrix.units} For each $w,w'\in G\act{}v$ define
$$e_{w,w'}\mydef p_{w,g_w}p_{w',g_{w'}}^*=p_{w,g_w(g_{w'})^{-1}}.$$
Then $\Mm_v\mydef \espan_R \{e_{w, w'} : w,w'\in G\act{}v\}$ is isomorphic to $M_{G\act{}v}$ via the map sending
$e_{w,w'}$ to $\eta_{w,w'}$.
\item \label{item.matrix.units.2} For each $w,w'\in G\act{}v$ and $g\in \stab{v}\mydef \{g\in G : g\act{}v=v\}$ define
$$e_{w,w'}^g\mydef p_{w,g_w}p_{v,g}p_{w',g_{w'}}^*=p_{w,g_wg(g_{w'})^{-1}}.$$
Then $e_{w,w'}^ge_{u,u'}^h=\delta_{w',u}e_{w,u'}^{gh}$ and $(e_{w,w'}^g)^*=e_{w',w}^{g^{-1}}$ for all $g,h\in G$ and $w,w',u,u'\in G\act{}v$.
\item \label{item.matrix.units.3} Suppose $G\act{}v$ is finite. For each $g\in \stab{v}$ define
$$W_v^g\mydef \sum_{w\in G\act{}v}e_{w,w}^g, \ \ \ (W_v^g)^*\mydef \sum_{w\in G\act{}v}e_{w,w}^{g^{-1}}.$$
Then $\Ww_v \  \mydef \ \espan_R \{W_v^g : g \in \stab{v}\}$ is a $^*$-algebra over $R$, and  $W_v^gW_v^h=W_v^{gh}$ for all $g,h \in \stab{v}$.
\end{enumerate}
\end{lemma}
\begin{proof}
First we prove \eqref{item.matrix.units}: Since $g_w^{-1}\act{}w=v=(g_{w'})^{-1}\act{}w'$, properties \ref{GET.b}--\ref{GET.c} of a $(G,E)$-family (Definition~\ref{def.LGE}) give
$$e_{w,w'}=
p_{w,g_w} p_{(g_{w'})^{-1}\act{w'}, (g_{w'})^{-1}}=\delta_{w,g_w\act{}((g_{w'})^{-1}\act{w'})}p_{w,g_w(g_{w'})^{-1}}=p_{w,g_w(g_{w'})^{-1}}.$$
Clearly $e_{w,w'}^*=e_{w',w}$ for $w,w'\in G\act{}v$. For $w,w',u,u'\in G\act{}v$ ,
$$e_{w,w'}e_{u,u'}=p_{w,g_w(g_{w'})^{-1}}p_{u,g_u(g_{u'})^{-1}}=\delta_{w,g_w(g_{w'})^{-1}\act{}u}p_{w,g_w(g_{w'})^{-1}g_u(g_{u'})^{-1}}.$$
Now, the equality $w=g_w(g_{w'})^{-1}\act{}u$ simplifies to $v=(g_{w'})^{-1}\act{}u$ or equivalently $w'=g_{w'}\act{}v=u$. So $e_{w,w'}e_{u,u'}=\delta_{w',u}p_{w,g_w(g_{u'})^{-1}}=\delta_{w',u}e_{w,u'}$. Hence $(e_{w,w'})_{w,w'\in G\act{}v}$ form matrix units with $e_{w,w}=p_{w,e_G}$ and $e_{w,v}=p_{w,g_w}$. By the universal property of $M_{G\act{}v}$ there exists a $^*$-algebra homomorphism $\pi\colon M_{G\act{}v}\to \Mm_v$ such that $\pi(\eta_{w,w'})=e_{w,w'}$.

The map $\pi$ is surjective by linearity. We prove $\pi$ is injective. Suppose that $\pi(\sum_{w,w'\in G\act{}v}r_{w,w'}\eta_{w,w'})=0$. For any $w',w''\in G\act{}v$,
$$r_{w',w''}e_{w',w''}=e_{w',w'}\Big(\sum_{w,u\in G\act{}v}r_{w,u}e_{w,u}\Big)e_{w'',w''}=0.$$
By Remark~\ref{nonzero} we know that $r_{w',w''}=0$, so $\pi$ is injective. Hence
$$M_{G\act{}v}\cong  \Mm_v.$$

To prove \eqref{item.matrix.units.2} we use that the collection $(e_{w,w'})_{w,w'\in G\act{}v}$ forms matrix units with $e_{w,w}=p_{w,e_G}$ and $e_{w,v}=p_{w,g_w}$. This gives the equalities $e_{w,w'}^g=e_{w,v}p_{v,g}e_{v,w'}$, $e_{v,w'}e_{u,v}=\delta_{w',u}e_{v,v}=\delta_{w',u}p_{v,e_G}$, and $p_{v,g}p_{v,e_G}p_{v,h}=p_{v,gh}$. It follows that
$$e_{w,w'}^ge_{u,u'}^h=e_{w,v}(p_{v,g}(e_{v,w'}e_{u,v})p_{v,h})e_{v,u'}=e_{w,v}\delta_{w',u}p_{v,gh}e_{v,u'}=\delta_{w',u}e_{w,u'}^{gh}.$$
Finally $e_{w,w'}^g=e_{w,v}e^g_{v,v}e_{v,w'}$, $e^g_{v,v}=p_{v,g}$ and $(p_{v,g})^*=p_{v,g^{-1}}$, so $(e_{w,w'}^g)^*=e_{w',w}^{g^{-1}}$.

The final property \eqref{item.matrix.units.3} follows from the computation
\begin{equation}
\begin{split}
W_v^gW_v^h&=\Big(\sum_{w\in G\act{}v}e_{w,w}^g\Big)\Big(\sum_{w'\in G\act{}v}e_{w',w'}^h\Big)=\sum_{w\in G\act{}v}\Big(\sum_{w'\in G\act{}v}e_{w,w}^ge_{w',w'}^h\Big),\\
&=\sum_{w\in G\act{}v}(e_{w,w}^ge_{w,w}^h)=\sum_{w\in G\act{}v}(e_{w,w}^{gh})=W_v^{gh}. \qedhere
\end{split}
\end{equation}
\end{proof}

\begin{defn}
\label{tensors}
Let $R$ be a ring. The \emph{algebraic tensor product} $A\algtensor  B$ of $^*$-algebras $A$ and $B$ over $R$ is the universal $^*$-algebra over $R$ generated by elements $\{a\otimes b :a\in A, b\in B\}$ subject to the relations
\begin{equation}
\begin{split}
&(a_1\otimes b_1)(a_2\otimes b_2)=a_1a_2\otimes b_1b_2,\\
&r(a_1\otimes b_1)=(ra_1)\otimes b_1=a_1\otimes (rb_1),\\
&(a_1\otimes b_1)^*=a_1^*\otimes b_1^*, \ \ a_1\otimes (b_1+b_2)=(a_1\otimes b_1)+(a_1\otimes b_2),\\
&(a_1+a_2)\otimes b_1=(a_1\otimes b_1)+(a_2\otimes b_1).
\end{split}
\end{equation}
Note that $A\algtensor B$ is spanned by the elements $\{a\otimes b :a\in A, b\in B\}$.
\end{defn}

\begin{prop}
\label{matrix.decomposition}
Let $(G,E,\varphi)$ be as in Notation~\ref{note}. Let $R$ be a unital commutative $^*$-ring. Suppose that  $V\subseteq E^0$ has the property that $E^0=\bigsqcup_{v\in V} G\act{}v$ is a partition of $E^0$. For $v\in V$, set $B_v\mydef \espan_R \{p_{w,g}: w\in G\act{}v, g\in G\}$.
\begin{enumerate}
\item \label{item.tensor.1} Each $B_v$ is a $^*$-algebra over $R$ and with $\Dd$ as in Theorem~\ref{thma}{A},
$$\Dd= \bigoplus_{v\in V} B_v.$$
\item \label{item.tensor.2} Suppose that $v\in V$ satisfies $|G\act{}v|<\infty$. With $\Ww_v, (e_{w,w'}^{g}), W_v^{g}$ as in Lemma~\ref{matrix.decomposition.pre}, there exists a surjective $^*$-algebra homomorphism
$$\psi_v\colon M_{G\act{}v}\algtensor \Ww_v \to B_v$$
sending $\eta_{w,w'}\otimes W_v^{g}$ to $e_{w,w'}^{g}$.
\item \label{item.tensor.3} Suppose that $G\act{}v$ is finite for each $v\in V$. Then each $\psi_v$ is an isomorphism. In particular
$$\Dd \cong \bigoplus_{v\in V} \left(M_{G\act{}v} \algtensor \Ww_v\right).$$
\end{enumerate}
\end{prop}
\begin{proof}
First we prove \eqref{item.tensor.1}: Fix $v\in V$. For any $g,g'\in G$ and any $w,w'\in G\act{}v$
\begin{equation}\label{projections}
\begin{split}
(p_{w,g})^*=p_{g^{-1}\act{}w,g^{-1}}\in B_v, \ \ \ p_{w,g}p_{w',g'}=\delta_{w,g\act{}w'}p_{w,gg'}\in B_v,
\end{split}
\end{equation}
so $B_v$ is a $^*$-algebra over $R$. Clearly $\Dd\mydef\espan_R\{p_{v,f}: v\in E^0, f\in G\}$ is the $R$-span of elements in $\bigcup_{v\in V} B_v$. Moreover, elements from two distinct algebras $B_v,B_{v'}$ have product zero, so $\Dd\cong \bigoplus_{v\in V} B_v$.

To prove \eqref{item.tensor.2} fix $v\in V$.  For any $a_i\mydef e_{w_i, w'_i} \in \Mm_v$, and any $b_j\mydef  W_v^{g_j}\in \Ww_v$ with $i,j\in \{1,2\}$ define $$a_i \Box b_j\mydef a_i b_j.$$
Since $a_ib_j=b_ja_i$ we get
\begin{equation*}
\begin{split}
&(a_1\Box b_1)(a_2\Box b_2)=a_1a_2\Box b_1b_2,\ \ r(a_1\Box b_1)=(ra_1)\Box b_1=a_1\Box (rb_1),\\
&(a_1\Box b_1)^*=a_1^*\Box b_1^*, \ \ a_1\Box (b_1+b_2)=(a_1\Box b_1)+(a_1\Box b_2),\\
&(a_1+a_2)\Box b_1=(a_1\Box b_1)+(a_2\Box b_1).
\end{split}
\end{equation*}
Notice that $a_i \Box b_j=e_{w, w'}W_v^{g_j}=e_{w_i,w'_i}^{g_j} \in B_v$. By $R$-linearity the above properties extend to any elements $a_i \in \Mm_v$ and $b_j \in \Ww_v$. Let $\varphi\colon \Mm_v\times \Ww_v\to \Mm_v\algtensor \Ww_v$ denote  the $R$-bilinear map $(a,b)\mapsto a\otimes b$. Define $h\colon \Mm_v\times \Ww_v \to B_v$ by $h(a,b)=a\Box b$. Since $h$ is $R$-bilinear the universal property of tensor products gives a unique $R$-linear homomorphism
$$\tilde{h}\colon  \Mm_v\algtensor \Ww_v \to B_v$$
satisfying $h=\tilde{h}\circ\varphi$. With $\stab{v}\mydef \{g\in G : g\act{}v=v\}$ as in Lemma~\ref{matrix.decomposition.pre} we have
\begin{equation}
B_v=\espan_R\{e_{w, w'}W_v^g : w,w'\in G\act{}v, g\in \stab{v} \}.
\end{equation}
Thus $\tilde{h}$ is surjective. Since $\tilde{h}$ is multiplicative, preserves adjoints on elementary tensors and is $R$-linear, it is multiplicative and $^*$-preserving on all of  $ \Mm_v\algtensor \Ww_v$. Hence $\tilde{h}$ is a surjective homomorphism of $^*$-algebras over $R$, sending $e_{w,w'}\otimes W_v^{g}$ to $e_{w,w'}\Box W_v^{g}$. The isomorphism $\Mm_v\cong M_{G\act{}v}$ of Lemma~\ref{matrix.decomposition.pre} now provides the desired map $\psi_v\colon M_{G\act{}v}\algtensor \Ww_v \to B_v$.

For \eqref{item.tensor.3}, fix $v\in V$. Each element $d\in B_v$ may be written as
$$d=\sum_{w,w'\in G\act{}v}\sum_{g\in F_{w,w'}} r_{(g,w,w')} e^g_{w,w'},$$
for some finite subsets $F_{w,w'}\subseteq \stab{v}$ and coefficients $r_{(g,w,w')}\in R$. Let $J\mydef\ker(\psi_v)$. Then
$J=M_{G\act{}v} \algtensor I\cong M_{G\act{}v}(I)$ where $I\subseteq \Ww_v$ is given by
\begin{equation*}
\begin{split}
I&\mydef \{a_{v,v} : a\in J\}\\
&=\Big\{a_{v,v}: \psi_v(a)=0, a=\sum_{w,w'\in G\act{}v}\sum_{g\in F_{w,w'}} r_{(g,w,w')} \, \eta_{w,w'}\otimes W_v^{g}\big\}\\
&=\Big\{a_{v,v}: \psi_v(a)=0, a=\sum_{w,w'\in G\act{}v} \eta_{w,w'}\otimes \sum_{g\in F_{w,w'}} r_{(g,w,w')} W_v^{g}\Big\}\\
&=\Big\{\sum_{g\in F_{v,v}} r_{(g,v,v)} W_v^{g} : \sum_{w,w'\in G\act{}v}\sum_{g\in F_{w,w'}} r_{(g,w,w')} e^g_{w,w'}=0\Big\}\\
&=\Big\{\sum_{g\in F_{v,v}} r_{(g,v,v)} W_v^{g} : \sum_{g\in F_{v,v}} r_{(g,v,v)} e^g_{v,v}=0\Big\}\\
&=\Big\{\sum_{g\in F_{v,v}} r_{(g,v,v)} W_v^{g} : \sum_{g\in F_{v,v}} r_{(g,v,v)} W_v^{g}=0\Big\}\\
&=\{0\}.
\end{split}
\end{equation*}
Hence $J=M_{G\act{}v} \algtensor I=\left\{0\right\}$, so $\psi_v\colon M_{G\act{}v}\algtensor \Ww_v \to B_v$ is injective, and hence an isomorphism.
\end{proof}

We are now in position to describe the $R$-algebra $\Dd$ appearing in the statement of the Graded Uniqueness Theorem~\ref{thma}{A}.

\begin{thm}
\label{thmb.new}
Let $(G,E,\varphi)$ be as in Notation~\ref{note}. Let $R$ be a unital commutative $^*$-ring. Fix  $V\subseteq E^0$ such that $E^0=\bigsqcup_{v\in V} G\act{}v$ is a partition of $E^0$. For each $v\in V$ let $\Ww_v$ be the $^*$-algebra of Lemma~\ref{matrix.decomposition.pre}. If $E$ is finite or for each $v\in V$, $G\act{}v$ is finite then
$$\Dd \cong \bigoplus_{v\in V} \left(M_{G\act{}v} \algtensor \Ww_v\right).$$
\end{thm}
\begin{proof}
Since each $G\act{}v, v\in V$ is finite (by assumption) the result follows from Proposition~\ref{matrix.decomposition}.
\end{proof}

\section{Proof of Theorem~\ref{thmb}{B}}
\label{section.three}

It was proved in \cite{MR3581326} that the $C^*$-algebra $\Oo_{G,E}$ is isomorphic to the groupoid $C^*$-algebra of the groupoid $\Gg_{\text{tight}}(\Ss_{G,E})$ as defined in Definition~\ref{def.GSEG}, cf.~\cite[Corollary~6.4]{MR3581326}. In this section we prove Theorem~\ref{thmb}{B}, which establishes an algebraic analogue of this $C^*$-algebraic result.

We need a number of preliminary results before proving the theorem. The proof of Theorem~\ref{thmb}{B} starts on page~\pageref{ThmB.proof}.


To make sense of the following Lemma~\ref{tech.lemma} we recall the notion of actions on the paths and on sets. Let $G$ be countable discrete group $G$, and $E$ a row-finite graph with no sources. An action $\sigma$ of $G$ on $E$ is a group homomorphism $g\mapsto \sigma_g$ from $G$ to the group of all automorphisms of $E$ (i.e., bijections $\sigma_g$ of $E^0\,\sqcup\, E^1$ such that $\sigma_g(E^i ) = E^i$, for $i=0,1$ and such that $r \circ \sigma_g = \sigma_g \circ r$, and $s \circ \sigma_g = \sigma_g \circ s$). An action $\sigma$ of $G$ on $E^*$ is a homomorphism $g\mapsto \sigma_g$ from $G$ to the group of bijections from $E^*$ to $E^*$. We often write $g\act{\alpha}$ instead of $\sigma_g(\alpha)$.
\begin{lemma}[{\cite[Proposition~2.4]{MR3581326}}]
\label{tech.lemma}
Let $(G,E,\varphi)$ be as in Notation~\ref{note}. Then $\sigma, \varphi$ extend to an action $\sigma \colon {G \times E^*}\to E^*$, $(g,\mu)\mapsto g\act{}\mu$ of $G$ on $E^*$ (viewed as a set) and a one-cocycle $\varphi \colon {G \times E^*}\to G$ for $\sigma$ such that:
\begin{enumerate}
\item $g\act{}(E^n)\subseteq E^n$,
\item $(gh)\act{}\alpha=g\act{}(h\act{}\alpha)$,
\item $\varphi(gh,\alpha)=\varphi(g, h\act{}\alpha)\varphi(h,\alpha)$, hence $\varphi(e_G,\alpha)=e_G$,
\item $\varphi(g,x)=g$,
\item $r(g\act{}\alpha)=g\act{}r(\alpha)$,
\item $s(g\act{}\alpha)=g\act{}s(\alpha)$,
\item $\varphi(g,\alpha)\act{}x=g\act{}x$,
\item \label{elfsim}$g\act{}(\alpha\beta)=(g\act{}\alpha)\varphi(g,\alpha)\act{}\beta$,
\item $\varphi(g,\alpha\beta)=\varphi(\varphi(g,\alpha),\beta)$.
\end{enumerate}
for all $g,h\in G$, $n\geq 0$, $x\in E^0$ and all $\alpha,\beta\in E^*$ with $s(\alpha)=r(\beta)$.
 \end{lemma}

Recall that a semigroup $\Ss$ 
is an inverse semigroup if for each $s\in \Ss$ there is a unique $s^*$ such that $s^*=s^*ss^*$ and $s=ss^*s$. A zero in $\Ss$ is an element $0\in \Ss$ such that $0s=s0=0$ for all $s\in \Ss$.

\begin{defn}[{\cite[Definition~2.4]{MR3581326}}]
\label{semigroup}
Let $(G,E,\varphi)$ be as in Notation~\ref{note}. Define
$$\Ss_{G,E}\mydef \{(\alpha,g,\beta) \in E^*\times G\times E^* : s(\alpha)=g\act{}s(\beta)\}\cup\{0\}.$$
\end{defn}
The proof of {\cite[Proposition~4.3]{MR3581326}} shows that under the multiplication
\begin{equation*}
  (\alpha ,g,\beta ) (\gamma ,h,\delta ) = \left \{
  \begin{matrix}
  (\alpha g\act{}\varepsilon ,\hfil \varphi (g,\varepsilon )h,\hfil \delta ), & \text {if } \gamma =\beta \varepsilon , \cr
  (\alpha ,\ g\varphi (h^{-1},\varepsilon )^{-1},\ \delta (h^{-1})\act{}\varepsilon ), & \text {if } \beta =\gamma \varepsilon , \cr
  0,\hfil \hfil \hfil & \text {otherwise.}
  \end{matrix}
  \right.
\end{equation*}
$\Ss_{G,E}$ is an inverse semigroup, in which $(\alpha,g,\beta)^*=(\beta,g^{-1},\alpha)$ and where $0$ acts as a zero in $\Ss_{G,E}$.

Recall that an idempotent in a semigroup is an element $s$ such that $s^2=s$. A semilattice is a partially ordered set $X$ such that each pair $s,t\in X$ has a greatest lower bound $s\wedge t$. Using the order on $\Ss_{G,E}$ given by $s\leq  t \Leftrightarrow s=ts^*s$, the set
\begin{equation}
\label{idemlattice}
\Ee\mydef\{(\alpha, e_G, \alpha) : \alpha \in E^*\}\cup\{0\}
\end{equation}
of all idempotents in $\Ss_{G,E}$ is a semilattice of mutually commuting elements with $s\wedge t=st$.

Let $X$ be any partially ordered set with minimum element $0$. A filter in $X$ is a nonempty subset $\xi\subseteq X$, such that $0\notin\xi$, if $x\in\xi$ and $x\leq y$, then $y\in\xi$, and if $x,y\in\xi$, there exists $z\in\xi$, such that $z\leq x,y$.

\begin{lemma} [{\cite[Definition~10.1, Theorem~12.9, Proposition~19.11]{MR2419901}}]
\label{tight.specum}
Let $(G,E,\varphi)$ be as in Notation~\ref{note}. Then
\begin{enumerate}
\item The semicharacter space (or the Stone spectrum of $\Ee$)
\begin{equation*}
\label{semicharacterspace}
\widehat{\Ee}\mydef \{\varphi \colon \Ee\to \{0,1\}:  \varphi(st)=\varphi(s)\varphi(t), \varphi\neq 0\}
\end{equation*}
with respect to the topology of pointwise convergence\footnote{Equivalently the relative topology from the product
space $\{0, 1\}^\Ee$.}  is a locally compact Hausdorff topological space.
\item The character space $\widehat{\Ee}_0\mydef \{\varphi \in \widehat{\Ee}:  \varphi(0)=0\}$ is a closed (hence locally compact) subset of $\widehat{\Ee}$. There is a bijection from $\widehat{\Ee}_0$ to the space of filters on $\Ee$ given by
\begin{equation*}
\label{filters}
\varphi \mapsto \xi_\varphi\mydef \{x\in \Ee :\varphi(x)=1 \}.
\end{equation*}
The inverse of this map bijection is given by $\xi\mapsto \varphi_\xi\mydef1_{\{x\in \xi\}}$.

\item Given an infinite path $x=e_1e_2e_3\dots\in E^\infty$, we define $\Ff_x\subseteq \Ee$ by $\Ff_x \mydef \{(e_1\dots e_n, e_G, e_1\dots e_n) : n\geq 1\}\cup \{(r(e_1),e_G, r(e_1))\}$.
Then the space $\widehat\Ee_{\infty}\mydef \{\varphi \in \widehat{\Ee}_0: \xi_\varphi \ \text{is an ultra-filter}\}$ is homeomorphic to the space $E^\infty$ (equiped with the product topology) of one-sided infinite paths on $E$ via the map $x\mapsto 1_{\Ff_x}$.
\item The tight spectrum $\widehat\Ee_{\operatorname{tight}}$ \cite[Definition~12.8]{MR2419901} in $\widehat{\Ee}_0$ satisfies
$$\overline{\widehat\Ee_{\infty}}=\widehat\Ee_{\operatorname{tight}}.$$
\end{enumerate}
\end{lemma}

As the following shows, for our setting the tight spectrum $\widehat\Ee_{\operatorname{tight}}$ corresponds to the infinite path space. We refer to \cite[p.~1074]{MR3581326} and \cite[Proposition~5.12]{ExePar} for special cases of this result.

\begin{lemma}[{\cite[Proposition~4.1]{ExeParSta}}]
\label{Etight.is.Einfty}
Let $(G,E,\varphi)$ be as in Notation~\ref{note}. Then $\widehat\Ee_{\text{tight}}=\widehat\Ee_{\infty}$.
\end{lemma}

Let $(G,E,\varphi)$ be as in Notation~\ref{note}. Recall that the canonical action of $g\in G$ on $x=e_1e_2\dots \in E^\infty$ is given by $x\mapsto g\act{}x$ where $g\act{}x$ is the unique infinite path such that $(g\act{}x)_1\dots(g\act{}x)_n=g\act{}(e_1\dots e_n)$ for all $n$. Identifying $\widehat\Ee_{\infty}\cong E^\infty$ the action of $\Ss_{G,E}$ on $E^\infty$ is given as follows: each $s=(\alpha,g,\beta)\in \Ss_{G,E}$ acts on elements of $Z(\beta)\mydef \{\beta x :  x\in s(\beta)E^\infty\}$ by $s\act{}(\beta x)\mydef\alpha (g\act{} x)$.
\begin{defn}[{\cite[Definition 4.6]{MR2419901}}]
\label{def.GSEG}
Let $(G,E,\varphi)$ be as in Notation~\ref{note}. We let $\Gg_{\text{tight}}(\Ss_{G,E})$ be the groupoid of germs of the action of $\Ss_{G,E}$ on $\widehat\Ee_{\operatorname{tight}}\cong\widehat\Ee_{\infty}\cong E^\infty$.
\end{defn}
As a set $\Gg_{\text{tight}}(\Ss_{G,E})$ is given by
\begin{equation}
\label{groupoid(G,E)}
\Gg_{\text{tight}}(\Ss_{G,E})\mydef \{[(\alpha,g,\beta), x]: (\alpha,g,\beta)\in \Ss_{G,E}, x\in Z(\beta)\},
\end{equation}
where for $s,t\in\Ss_{G,E}$, $[s,x]=[t,y]$ if and only if $x=y$ and there exists nonzero $e=(\gamma,e_G,\gamma)\in \Ee$ such that $e\act{}x=x$ (or equivalently $x\in Z(\gamma)$) and $se=te$. The unit space
$$\Gg_{\text{tight}}(\Ss_{G,E})^{(0)}=\{[(\alpha,e_G,\alpha), x]:  x\in Z(\alpha)\},$$
is identified with $E^\infty$ via$[(\alpha,e_G,\alpha), x]\mapsto x$, and the source and range maps are given by $s([(\alpha,g,\beta), x])=x$ and $r([(\alpha,g,\beta), \beta y])=\alpha (g\act{}y)$. The groupoid $\Gg_{\text{tight}}(\Ss_{G,E})$ is equiped with a topology making it locally compact and ample with a Hausdorff unit space, cf.~\cite[Proposition~4.14]{MR2419901}.

\begin{defn}
\label{Steinberg}
Let $\Gg$ be a locally compact and ample groupoid with Hausdorff unit space and $R$ a unital commutative $^*$-ring. The \emph{Steinberg algebra} $A_R(\Gg)$ is defined as $$A_R(\Gg)\mydef\espan_R\{1_{B} : B \ \text{is a compact open bisection}\}\subseteq R^\Gg.$$
We endow $A_R(\Gg)$ with pointwise  addition. Multiplication and adjoint are given by
$fg(\gamma)\mydef \sum_{\alpha\beta=\gamma}f(\alpha)g(\beta)$ and $f^*(\gamma)\mydef f(\gamma^{-1})^*$.
\end{defn}
When $\Gg$ is Haurdorff $A_R(\Gg)$ is just the $^*$-algebra of locally constant functions with compact support from $\Gg$ to $R$.

Let $(G,E,\varphi)$ be as in Notation~\ref{note}. By \cite[Proposition~4.18]{MR2419901}, for each $s\in \Ss_{G,E}$ the set $\Theta_s\mydef \{[s, x]: x\in Z(\beta)\}$ is a compact open bisection and such sets form a basis for the topology on $\Gg_{\text{tight}}(\Ss_{G,E})$. It follows that
$$A_R(\Gg_{\text{tight}}(\Ss_{G,E}))=\espan_R\{1_{\Theta_s} : s\in \Ss_{G,E}\}.$$

We now prove that every Exel--Pardo $^*$-algebra admits an homomorphism into a Steinberg algebra.
\begin{prop}
\label{nonzero}
Let $(G,E,\varphi)$ be as in Notation~\ref{note}. Let $R$ be a unital commutative $^*$-ring. Then there exists a unique $^*$-homomorphism $\pi_{G,E}\colon \Ll_R(G,E) \to A_R(\Gg_{\text{tight}}(\Ss_{G,E}))$ such that
\begin{eqnarray*}
\pi_{G,E}(p_{v,f})=1_{\Theta_{(v,f,f^{-1}\act{}v)}} \ \ \text{and} \ \ \pi_{G,E}(s_{e,g})=1_{\Theta_{(e,g,g^{-1}\act{}s(e))}},
\end{eqnarray*}
for any $v\in E^0$, $e\in E^1$ and $f,g\in G$. In particular, the generators $\{p_{v,f}, s_{e,g}\}$ of $\Ll_R(G,E)$ are nonzero.
\end{prop}
\begin{proof}
Using \cite[Proposition~7.4]{MR2419901} and \cite[Proposition~4.5(3)]{MR2565546} we see that $1_{\Theta_s}1_{\Theta_t}=1_{\Theta_{st}}$ and $(1_{\Theta_s})^*=1_{\Theta_{s}^{-1}}=1_{\Theta_{s^*}}$ for $s,t\in \Ss_{G,E}$. A tedious but straightforward computation using Lemma~\ref{tech.lemma} shows that the elements
$$P_{v,f}\mydef 1_{\Theta_{(v,f,f^{-1}\act{}v)}} \ \ \text{and} \ \ S_{e,g}\mydef 1_{\Theta_{(e,g,g^{-1}\act{}s(e))}}$$
form a $(G,E)$-family in the Steinberg algebra $A_R(\Gg_{\text{tight}}(\Ss_{G,E}))$. The result now follows from the universal property of $\Ll_R(G,E)$ (Theorem~\ref{thm.univarsal}).
\end{proof}

\begin{remark}
\label{the.remark}
We do not assert that the map of Proposition~\ref{nonzero} is injective.
In the unital case, an interesting approach was presented in \cite[Theorem~6.4]{ClaExePar} intending to show that the map $\pi_{G,E}$ in Proposition~\ref{nonzero} is an isomorphism, but we believe this argument is valid only for $R=\CC$. The idea was to build an inverse
$$\phi\colon A_R(\Gg_{\text{tight}}(\Ss_{G,E}))\to \Oo_{(G,E)}^\text{alg}(R)$$
using the construction of a map $\psi\colon C^*(\Gg_{\text{tight}}(\Ss_{G,E}))\to C^*_{\text{tight}}(\Ss_{G,E})$ from the proof of \cite[Theorem~2.4]{MR2644910}. We outline relevant details:

Let $\Ss$ be an inverse semigroup with zero. A representation $\sigma$ of $\Ss$ in a unital $C^*$-algebra $A$ is a zero- and $^*$-preserving multiplicative map $$\sigma:\Ss\to A.$$ If $A=B(H)$, then $\sigma$ is called a representation of $\Ss$ on $H$ (\footnote{The zero-preserving property does not appear in \cite[Definition~10.4]{MR2644910}, but presumably $\sigma_0=0$ was intended for inverse semigroups with $0$.}). For $e\in \Ee\mydef \{s\in \Ss: s^2=s\}$, we define $D_e\mydef\{\varphi \in \widehat\Ee : \varphi(e)=1\}$. This is a clopen subset of $\widehat\Ee$. For each representation $\sigma\colon \Ss\to B(H)$ there exists a unique $^*$-homomorphism $\pi_\sigma\colon C_0(\widehat\Ee) \to B(H)$ such that for each  $e\in \Ee$, $\pi_\sigma(1_{D_e})=\sigma_e$ (\cite[Proposition~10.6]{MR2419901}). It is a tight representation precisely if $\pi_\sigma$ vanishes on
$C_0(\widehat\Ee\setminus\smallspace \widehat\Ee_{\operatorname{tight}})$ (\cite[Theorem~13.2]{MR2419901}). The tight $C^*$-algebra of $\Ss$, denoted $C^*_{\text{tight}}(\Ss)$, is defined as the universal $C^*$-algebra generated by a universal tight representation $\pi_u$ of $\Ss$, so any tight representation $\pi$ of $\Ss$ in a unital $C^*$-algebra $A$ induces a unital $^*$-homomorphism $$\psi\colon C^*_{\text{tight}}(\Ss)\to A$$ such that $\psi\circ \pi_u=\pi$. Viewing $C^*_{\text{tight}}(\Ss)$ as an algebra of operators on a Hilbert space $H$ via some faithful representation, $\pi_u\colon \Ss \to C^*_{\text{tight}}(\Ss)$ may be regarded as a tight representation of $\Ss$ on $H$. Since $\pi_u$ is tight $\pi_{(\pi_u)}$ factors through $C_0(\widehat\Ee_{\operatorname{tight}})$ giving a representation $\pi$ of  $C_0(\widehat\Ee_{\operatorname{tight}})$ on $H$. Denoting $R:C_0(\widehat\Ee) \to C_0(\widehat\Ee_{\operatorname{tight}})$  for the restriction map, we get the following commuting diagram:
\begin{displaymath}
  \xymatrix{
    {C_0(\widehat\Ee)} \ar[rr]^{\pi_{(\pi_u)}} \ar[dr]^{R}
    && {B(H)}\\
    & {C_0(\widehat\Ee_{\operatorname{tight}}).} \ar[ur]^{\pi}
    &
  }
\end{displaymath}
Define $\Gg\mydef \Gg_{\text{tight}}(\Ss_{G,E})$ and $D_e^{\widehat\Ee_{\text{tight}}}\mydef\{\varphi \in \widehat\Ee_{\text{tight}} : \varphi(e)=1\}$. Since the collection $\{\Theta_s: s\in \Ss_{G,E}\}$ is a basis for the topology on $\Gg$ we may define $\pi_u \times \pi\colon C_c(\Gg)\to B(H)$ via
\begin{align*}
(\pi_u \times \pi)(1_{\Theta_s})&\mydef\pi_u(s)\pi(1_{D_{s^*s}^{\widehat\Ee_{\text{tight}}}})\\
&=\pi_u(s)\pi_{\pi_u}(1_{D_{s^*s}})=\pi_u(s)\pi_u(s^*s)=\pi_u(s).
\end{align*}
and extend by linearity. The map $\pi_u \times \pi$ is well-defined (see \cite[Lemma~8.4]{MR2419901} but we will give a sketch below) and norm decreasing (see the proof of \cite[Theorem~8.5]{MR2419901}). Its range is contained in $C^*_{\text{tight}}(\Ss)$, so it extends to a homomorphism $\psi\colon C^*(\Gg_{\text{tight}}(\Ss_{G,E}))\to C^*_{\text{tight}}(\Ss_{G,E})$.
Restricting this map to the Steinberg algebra and identifying its image we obtain
$$\phi\colon A_\CC(\Gg_{\text{tight}}(\Ss_{G,E}))\to \Oo_{(G,E)}^\text{alg}(\CC).$$

To see why $\pi_u \times \pi$ is well-defined, let $\tilde\pi$ denote the weakly continuous extension of $\pi$ to the algebra $\Bb(\widehat\Ee)$ of all bounded Borel measurable functions on $\widehat\Ee$.
Let $s\colon \Gg\to \Gg^{(0)}$ be the source map and fix $\xi, \eta\in H$. For each $s\in \Ss_{G,E}$ a finite Borel measure $\mu_s\mydef\mu_{s,\xi, \eta}$ on $\Theta_s$ is defined (in the proof of \cite[Lemma~8.4]{MR2419901}) by
$$\mu_s(A)\mydef \langle \pi_u(s)\tilde\pi(1_{s(A)})\xi,\eta\rangle $$
for every Borel measurable $A\subseteq \Theta_s$. Fix $f\mydef\sum_{s\in J} c_s1_{\Theta_s} \in C_c(\Gg)$ with $c_s\in \CC$ and $|J|<\infty$. Let $M \mydef \bigcup_{s\in J} \Theta_s$ and let $\mu$ be a measure on $M$ such that $\mu(A)=\mu_s(A)$ for every $s\in \Ss_{G,E}$ and every measurable $A\subseteq \Theta_s$. Since
$$\Big\langle \sum_{s\in J} c_s\pi_u(s)\xi,\eta\Big\rangle=\int_Mf d\mu,$$
it follows that if $f=0$ in $C_c(\Gg)$ then $\sum_{s\in J} c_s\pi_u(s)=0$ in $B(H)$. Therefore $\pi_u \times \pi$ is well-defined. The point is that constructing the inverse $\phi$ uses Hilbert space arguments. As such it is less clear if the map $\pi_{G,E}$ in Proposition~\ref{nonzero} has an inverse for $R\neq \CC$.
\end{remark}

When $\Gg_{\text{tight}}(\Ss_{G,E})$ is Hausdorff the result of Proposition~\ref{nonzero} can be substantially improved, resulting in Theorem~\ref{thmb}{B}. Before giving the proof of Theorem~\ref{thmb}{B} we need two preliminary lemmas.

\begin{lemma}
\label{lem.tU}
Let $(G,E,\varphi)$ be as in Notation~\ref{note}. Let $R$ be a unital commutative $^*$-ring. Let $\Bb$ be the set of compact open bisections of $\Gg_{\text{tight}}(\Ss_{G,E})$.
Let $$\Jj\mydef \left\{J\subseteq \Ss_{G,E} : |J|<\infty, \bigsqcup_{s\in J} \Theta_s=\bigcup_{s\in J} \Theta_s \in \Bb\right\}.$$
Then
\begin{enumerate}
\item For each $U\in \Bb$ there exists $J\in \Jj$ such that $U=\bigsqcup_{s\in J} \Theta_s$.
\item For each $J\in \Jj$ define $t^J\mydef \sum_{(\alpha,g,\beta)\in J} s_{\alpha,g}s_{\beta,e_G}^*$. Then
\begin{align*}
\forall \ I,J\in \Jj:\ \  \bigsqcup_{s\in J}\Theta_{s}=\bigsqcup_{s\in I}\Theta_{s} \ \Longrightarrow \  t^J=t^I.
\end{align*}
\item For each $U\in \Bb$ there exists a unique $t_U\in \Ll_R(G,E)$ such that $t_U=\sum_{(\alpha,g,\beta)\in J} s_{\alpha,g}s_{\beta,e_G}^*$ whenever $U=\bigsqcup_{s\in J}\Theta_s$ and $|J|<\infty$.
\end{enumerate}
\end{lemma}
\begin{proof}
(1) For $s=(\alpha, g, \beta)\in \Ss_{G,E}$,  $\Theta_s\mydef \{[(\alpha,g,\beta), x]: x\in Z(\beta)\}$ is a compact open bisection and such sets form a basis for the topology on $\Gg\mydef\Gg_{\text{tight}}(\Ss_{G,E})$. We prove that for each $U\in \Bb$ there exists a finite $J\subseteq  \Ss_{G,E}$ such that $U=\bigsqcup_{s\in J} \Theta_s$ is a disjoint union:
Compactness of $U$ gives $U=\bigcup_{t\in J} \Theta_t$ for some finite $J\subseteq \Ss_{G,E}$. Since $s(\Theta_{(\alpha,g,\beta)})=Z(\beta)$, any two sets from $\{s(\Theta_t): t\in J$\} are either disjoint or contained in one another. Since $U$ is a bisection $s|_U$ is a homeomorphism, so any two sets from $\{\Theta_t: t\in J$\} are either disjoint or contained in one another. Now remove the superfluous $t\in J$.

(2) To ease terminology, we view $J\in \Jj$ as a partition of $U$ via  the equality $U=\bigsqcup_{s\in J}\Theta_s$. We claim that for two partitions $I,J$ of the same $U\in \Bb$ we have $t^J=t^I$. The prove strategy is first to show that if $\{\Theta_{(\alpha, g, \beta)} : (\alpha, g, \beta)\in J\}$ is a partition of $U\in \Bb$ and if $n\geq \max_{(\alpha, g, \beta)\in J}|\beta|$ then the refinement $$J'\mydef \{(\alpha g\act{}\lambda, \varphi(g,\lambda), \beta\lambda) : (\alpha,g,\beta)\in J, \lambda\in s(\beta)E^{(n-|\beta|)}\}$$ of $U$ satisfies $t^J=t^{J'}$. One then shows that if $J,I$ are two partitions of the same $U\in \Bb$, then for large enough $n$ the refinements $J',I'$ just described above are equal.

Following this strategy, take any nonempty $U\in \Bb$ and a partition $U=\bigsqcup_{s\in J} \Theta_s$ of $U$. Fix any
$$n\geq\max_{(\alpha, g, \beta)\in J}|\beta|.$$
For each $r=(\alpha, g, \beta)\in J$ and each $m\in \NN$ calculations
%
%
show that
\begin{align*}
\Theta_r&=\Theta_{r}\Theta_{r^*r}=\Theta_{r}\Theta_{(\beta, e_G,\beta)}\\
&=\Theta_{r}\bigsqcup_{\lambda\in s(\beta)E^m}\Theta_{(\beta\lambda, e_G,\beta\lambda)}=\bigsqcup_{\lambda\in s(\beta)E^m}\Theta_{(\alpha g\act{}\lambda, \varphi(g,\lambda), \beta\lambda)}
\end{align*}
and similarly
\begin{align*}
s_{\alpha,g}s_{\beta,e_G}^*&=\sum_{\lambda\in s(\beta)E^m}s_{\alpha g\act{}\lambda, \varphi(g,\lambda)}s_{\beta\lambda,e_G}^*.
\end{align*}
Hence $t^{\{r\}}=t^{\{{(\alpha g\act{}\lambda, \varphi(g,\lambda), \beta\lambda) : \lambda\in s(\beta)E^m\}}}$. In particular, we also have that $t^{\{r\}}=t^{\{{(\alpha g\act{}\lambda, \varphi(g,\lambda), \beta\lambda) : \lambda\in s(\beta)E^{(n-|\beta|)}\}}}$. With
$$J'\mydef \{(\alpha g\act{}\lambda, \varphi(g,\lambda), \beta\lambda) : (\alpha,g,\beta)\in J, \lambda\in s(\beta)E^{(n-|\beta|)}\}$$ it follows that
$$t^J=\sum_{s\in J}t^{\{s\}}=\sum_{s\in J}t^{\{(\alpha g\act{}\lambda, \varphi(g,\lambda), \beta\lambda): \, s=(\alpha,g,\beta), \lambda\in s(\beta)E^{(n-|\beta|)}\}}=t^{J'}.$$

Now suppose $U=\bigsqcup_{s\in I} \Theta_s$ is another partition of $U$. Let $n\mydef \max_{(\alpha, g, \beta)\in J\cup I}|\beta|$. By the above, $J'\mydef \{(\alpha g\act{}\lambda, \varphi(g,\lambda), \beta\lambda) : (\alpha,g,\beta)\in J, \lambda\in s(\beta)E^{(n-|\beta|)}\}$ and $I'\mydef \{(\gamma h\act{}\lambda, \varphi(h,\lambda), \delta\lambda) : (\gamma ,h,\delta )\in I, \lambda\in s(\delta)E^{(n-|\delta|)}\}$ satisfy $t^J=t^{J'}$ and $t^I=t^{I'}$.
We have
$$s(U)=\bigsqcup_{\{\nu\in E^n : (\mu, f, \nu)\in J'\}}Z(\nu)=\bigsqcup_{\{\nu\in E^n : (\mu, f, \nu)\in I'\}}Z(\nu).$$
Hence $\{\nu : (\mu, f, \nu)\in J'\}=\{\nu : (\mu, f, \nu)\in I'\}$. Since $U$ is a bisection, we deduce that $I'=J'$, hence $t^J=t^{J'}=t^{I'}=t^I$.

(3) For each $U\in \Bb$ use (1) to find $I\in \Jj$ such that $U=\bigsqcup_{s\in I} \Theta_s$. The desired result now follows from property (2).
\end{proof}

Let $\Gg$ be a locally compact, ample groupoid with a Hausdorff unit space. Let $\Bb$ be the family of all compact open bisections of $\Gg$. Let $R$ be a unital commutative $^*$-ring and let $B$ be a $^*$-algebra over $R$. For Hausdorff groupoids (regarded as carrying the trivial grading), \cite[Definition 3.10]{MR3274831}(\footnote{Condition (R3) in \cite{MR3274831} is missing ``such that $B \cup D$ is a bisection''.}) defines a \emph{representation of $\Bb$} in $B$ as a family $\{t_U: U\in \Bb\}\subseteq B$ satisfying
\begin{enumerate}\renewcommand{\theenumi}{R\arabic{enumi}}
\item\label{item:zero} $t_\emptyset = 0$;
\item\label{item:multiplicative} $t_Ut_V = t_{UV}$ for all $U,V \in \Bb$; and
\item\label{item:additive} $t_U + t_V = t_{U \cup V}$ whenever $U$ and $V$ are disjoint elements of $\Bb$ such that $U \cup V$ is a bisection.
\end{enumerate}

For non-Hausdorff $\Gg$, it is not clear that this is an appropriate definition of a representation, but we will nevertheless want to refer to \eqref{item:zero}--\eqref{item:additive} in this context.

\begin{lemma}
\label{rep.tU}
Let $(G,E,\varphi)$ be as in Notation~\ref{note}. Let $R$ be a unital commutative $^*$-ring. Then the family $\{t_U : U\in \Bb\}$ of Lemma~\ref{lem.tU} satisfies \eqref{item:zero}--\eqref{item:additive}.
\end{lemma}

\begin{proof}
Property (R1) follows from $t_\emptyset = t^{\emptyset}=0$.

For (R3), take any pair of disjoint sets $U,V \in \Bb$ such that $U \cup V$ is a bisection. Choose $J,I\in \Jj$ such that $U=\bigsqcup_{s\in J}\Theta_s$ and $V=\bigsqcup_{s\in I}\Theta_s$. Then $t_U=t^J$ and $t_V=t^I$. Since $U\sqcup V=\bigsqcup_{s\in I\sqcup J}\Theta_s$ we get $t^{I\sqcup J}=t_{U\sqcup V}$, so $$t_U+t_V= t^{I}+t^{J}=t^{I\sqcup J}=t_{U\sqcup V}.$$

Finally we consider (R2). Fix any $U=\bigsqcup_{s\in J} \Theta_s$, $V=\bigsqcup_{t\in I} \Theta_t$ in $\Bb$.
If $(s,t)$, $(s',t')$ are distinct elements of $J\times I$, then $\Theta_{s}\Theta_{t}\cap \Theta_{s'}\Theta_{t'}=\emptyset$ because both $U$ and $V$ are bisections. So $UV=\bigsqcup_{s\in J,  t\in I}\Theta_s\Theta_t$. Since $\Theta_{s}\Theta_{t}=\Theta_{st}$ for any $s,t\in\Ss_{G,E}$ (see \cite[Proposition~7.4]{MR2419901}), we get $UV=\bigsqcup_{s\in J,  t\in I}\Theta_{st}$. Using the relations in $\Ll_R(G,E)$ we have $t_{\Theta_{s}}t_{\Theta_{t}}=t_{\Theta_{st}}$ for any $s\in J, t\in I$. Hence
$$t_{UV}=\sum_{s\in J, t\in I}t_{\Theta_{st}}=\sum_{s\in J, t\in I}t_{\Theta_{s}}t_{\Theta_{t}}=\Big(\sum_{s\in J}t_{\Theta_{s}}\Big)\Big(\sum_{t\in I}t_{\Theta_{t}}\Big)=t_Ut_V.\qedhere$$
\end{proof}

Having constructed a homomorphism from the Exel--Pardo $^*$-algebra to the Steinberg algebra of $\Gg_{\text{tight}}(\Ss_{G,E})$ and a representation of $\Bb$ in $\Ll_R(G,E)$ we can now prove Theorem~\ref{thmb}{B}.

\begin{proof}[Proof of Theorem~\ref{thmb}{B}]
\label{ThmB.proof}
Let $\Bb$ denote the set of compact open bisections of $\Gg_{\text{tight}}(\Ss_{G,E})$. By Lemma~\ref{rep.tU} the family $\{t_U : U\in \Bb\}$ of Lemma~\ref{lem.tU} is a representation of $\Bb$ in $\Ll_R(G,E)$. So \cite[Proposition~2.3]{MR3372123} shows that $\pi_{G,E}$ is an $R$-algebra isomorphism. One then verify the map is $^*$-preserving (cf.~\cite[Theorem~3.11]{MR3274831}).
\end{proof}

\begin{remark}
\label{rem.pi.map}
In Theorem~\ref{thmb}{B} the inverse map $\pi^{-1}_{G,E}$ from $A_R(\Gg_{\text{tight}}(\Ss_{G,E}))$ into $\Ll_R(G,E)$ satisfies
\begin{equation}
\label{defofpiinv}
\pi^{-1}_{G,E}(1_{\Theta_{(\alpha, g, \beta)}})=t_{\Theta_{(\alpha, g, \beta)}}, \ \ \text{where}\ \ t_{\Theta_{(\alpha, g, \beta)}}\mydef s_{\alpha,g}s_{\beta,e_G}^*.\end{equation}
We remark that for \emph{any} triple $(G,E,\varphi)$ as in Notation~\ref{note}, $\pi_{G,E}$ is a $^*$-isomorphism if and only if \eqref{defofpiinv} extends by linearity to a well-defined $R$-linear map on $A_R(\Gg_{\text{tight}}(\Ss_{G,E}))$.
\end{remark}

Due to work in \cite{ExeParSta} it is known when the groupoid $\Gg_{\text{tight}}(\Ss_{G,E})$ is Hausdorff. We recall the relevant terminology.
A path $\alpha\in E^*$ is \emph{strongly fixed} by $g\in G$ if $g\act{}\alpha=\alpha$ and $\varphi(g,\alpha)=e_G$. In addition if no prefix (i.e., initial segment) of $\alpha$ is strongly fixed by $g$ we say $\alpha$ is a \emph{minimal} \emph{strongly fixed} path for $g$ (\cite[Definition~5.2]{MR3581326}).

\begin{prop}[{\cite[Theorem~4.2]{ExeParSta}}]\label{iff.Hausdorff}
\label{when.hausdorff}
Let $(G,E,\varphi)$ be as in Notation~\ref{note}. Then the following properties are equivalent:
\begin{enumerate}
\item \label{iff.Hausdorff.1} For every $g\in G$, and every $v\in E^0$ there are at most finitely many minimal strongly fixed paths for $g$ with range $v$.
\item \label{iff.Hausdorff.2}  $\Gg_{\text{tight}}(\Ss_{G,E})$ is Hausdorff.
\end{enumerate}
\end{prop}

Combining Proposition~\ref{when.hausdorff} and Theorem~\ref{thmb}{B} we get:

\begin{cor}
Let $(G,E,\varphi)$ be as in Notation~\ref{note}. Let $R$ be a unital commutative $^*$-ring. Suppose that for every $g\in G$, and every $v\in E^0$ there are at most finitely many minimal strongly fixed paths for $g$ with range $v$. Then
$$\Ll_R(G,E)\cong A_R(\Gg_{\text{tight}}(\Ss_{G,E})).$$
\end{cor}

\subsection{Simplicity and pure infiniteness for $\Ll_R(G,E)$}

Having a Steinberg algebra realisation of Exel--Pardo $^*$-algebras we can use known results on Steinberg algebras to say something about $\Ll_R(G,E)$. In particular the results in \cite{MR3581326, ExeParSta} apply to our setting giving the two propositions below.

The terminology used in Proposition~\ref{simple.LGE} and Proposition~\ref{pi.LGE} was introduced in \cite{MR3581326}. More specifically, for the definition of a weakly-$G$-transitive directed graph $E$, the notion of a $G$-circuit in $E$ having an entry and the definiton of a group element $g\in G$ being slack at a vertex $v$, see \cite[Definition~13.4]{MR3581326}, \cite[Definition~14.4]{MR3581326} and \cite[Definition~14.9]{MR3581326} respectively.

\begin{prop}[{cf.~\cite[Theorem 4.5]{ExeParSta}}]
\label{simple.LGE}
Let $(G,E,\varphi)$ be as in Notation~\ref{note}. Let $R$ be a unital commutative $^*$-ring. Suppose that for every $g\in G$, and every $v\in E^0$ there are at most finitely many minimal strongly fixed paths for $g$ with range $v$. Then $\Ll_R(G,E)$ is simple if and only if $R$ is simple and
\begin{enumerate}
\item the graph $E$ is weakly-$G$-transitive;
\item every $G$-circuit has an entry; and
\item for every vertex $v$, and every $g\in G$ that fixes $Z(v)$ pointwise, $g$ is
slack at $v$.
\end{enumerate}
\end{prop}

\begin{prop}[{cf.~\cite[Theorem 4.7]{ExeParSta}}]
\label{pi.LGE}
Let $(G,E,\varphi)$ be as in Notation~\ref{note}. Let $R$ be a unital commutative $^*$-ring. Suppose that for every $g\in G$, and every $v\in E^0$ there are at most finitely many minimal strongly fixed paths for $g$ with range $v$ and that $\Ll_R(G,E)$ is simple. If $E$ contains at least one $G$-circuit, then $\Ll_R(G,E)$ is purely infinite (simple).
\end{prop}
\begin{proof}
Use translates of a $G$-circuit to construct a infinite path and proceed as in \cite{ExeParSta}.
\end{proof}

\section{Proof of Theorem~\ref{thmc}{C}}
\label{section.four}

In this section we prove Theorem~\ref{thmc}{C} studying Steinberg algebras of non-Hausdorff groupoids. In this setting it is not clear when the $^*$-homomorphism $\pi_{G,E}\colon \Ll_R(G,E) \to A_R(\Gg_{\text{tight}}(\Ss_{G,E}))$ of Proposition~\ref{nonzero} is a $^*$-isomorphism. We still know that the family $\{t_U : U\in \Bb\}$ of Lemma~\ref{lem.tU} satisfies \eqref{item:zero}--\eqref{item:additive}, but we cannot conclude  immediately that $\pi_{G,E}$ admits an inverse.
By considering actions with an appropriate amount of ``strongly fixed'' paths one can nevertheless get an inverse. We now introduce such paths and state Theorem~\ref{non.Hasu.EG}, giving Theorem~\ref{thmc}{C} as a corollary.
\begin{defn}
\label{strongly.fixed}
Let $E$ be as in Notation~\ref{note}. Let $\beta\in E^*\setminus E^0$ be a finite path in $E$. We say $\beta$ is \emph{strongly fixed} if $\beta$ is strongly fixed by some $g\in G\setminus\{e_G\}$. We say $\beta$ is \emph{minimal} strongly fixed if no prefix (i.e., initial segment) of $\beta$ is strongly fixed. Let $x\in E^\infty$ be an infinite path in $E$. We say $x$ is \emph{strongly fixed} if some initial segment $\beta\in  r(x)E^*\setminus\{r(x)\}$ of $x$ is strongly fixed.
\end{defn}

\begin{thm}
\label{non.Hasu.EG}
Let $(G,E,\varphi)$ be as in Notation~\ref{note}. Let $R$ be a unital commutative $^*$-ring. For $u\in E^0$, let $\mathcal{F}_u$ be the set of all minimal strongly fixed paths with range $u$. Suppose that $Z(\gamma)\cap Z(\gamma')=\emptyset$ whenever $\gamma\neq\gamma'\in \mathcal{F}_u$. If for each $u\in E^0$,
\begin{itemize}
\item there exist $x\in Z(u)$ that is not strongly fixed, or
\item $\mathcal{F}_u$ is finite,
\end{itemize}
then
$\Ll_R(G,E) \cong A_R(\Gg_{\text{tight}}(\Ss_{G,E}))$,
via the map $\pi_{G,E}$ of Proposition~\ref{nonzero}.
\end{thm}

The proof of Theorem~\ref{non.Hasu.EG} is essentially contained in the five lemmas Lemma~\ref{non-hausdorff}--Lemma~\ref{final110paths.new}. Lemma~\ref{non-hausdorff} establishes a graded structure of $A_R(\Gg_{\text{tight}}(\Ss_{G,E}))$ allowing us to use the Graded Uniqueness Theorem~\ref{thma}{A}. Lemma~\ref{Um=Uniff} further reduces the problem, so we only need to prove injectivity of $\pi_{G,E}$ on each  $\espan_R\{p_{u,g}, g\in G\}$. We then consider two complementary cases:
\begin{enumerate}
\item\label{caseONE} There is an infinite path with range $u$ that is not strongly fixed.
\item\label{caseTWO} All infinite paths with range $u$ are strongly fixed.
\end{enumerate}
In case \eqref{caseONE} we prove that the elements  $\{p_{u,g}, g\in G\}$ are linearly independent (Lemma~\ref{with111paths}). In case \eqref{caseTWO} we introduce a certain disjointification of $p_{u,g}$ and $\pi_{G,E}(p_{u,g})$ relative to suitable strongly fixed paths (Lemma~\ref{with110paths.new}). We use this disjointification to show that the elements $\{p_{u,g}, g\in G\}$ are ``sufficiently'' linearly independent (Lemma~\ref{final110paths.new}). We finally combine these results in the proof of Theorem~\ref{non.Hasu.EG}.

Before getting more technical we present two examples of non-Hausdorff groupoids of germs $\Gg_{\text{tight}}(\Ss_{G,E})$ illustrating how these two complementary cases may arise.

\begin{example}
\label{exp.1}
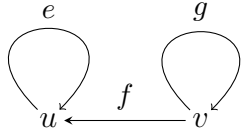
\begin{figure}
\centering
\[
\begin{tikzpicture}
    \node[circle, inner sep=1pt] (v1) at (0,0) {$u$};
    \node[circle, inner sep=1pt] (v2) at (2,0) {$v$};
	\path[->,every loop/.style={looseness=18}] (v1)
	         edge[in=45,out=135,loop, black]  node[above, pos=0.5,black]{\small$e$} ();
	\path[->,every loop/.style={looseness=18}] (v2)
	         edge[in=45,out=135,loop, black]  node[above, pos=0.5,black]{\small$g$} ();
    \draw[black, -stealth, in=0, out=180] (v2) to node[above, pos=0.5, black] {\small$f$} (v1);
\end{tikzpicture}
\]
\caption{Smallest graph of a non-Hausdorff groupoid $\Gg_{A,B}$.} \label{fig.non.haus.1}
\end{figure}
Let $\Gg_{A,B}$ be the groupoid of germs for the Katsura triple $(\ZZ,E,\varphi)$ associated to the matrices
$$A=\minimatrix{1}{1}{0}{1}, \ \ \ B=\minimatrix{1}{0}{0}{1}$$
(see Definition~\ref{def.graph.for.AB} and Definition~\ref{def.GSEG}). The graph $E$  associated to $A$ is illustrated on Figure~\ref{fig.non.haus.1}.

Notice that $\Gg_{A,B}$ is non-Hausdorff:
With $N$, $A$, $B$, $I$ and $E$ as in Definition~\ref{def.graph.for.AB} and with
$K^{\alpha,l}_j\mydef \frac{l B_{r(\alpha_1)s(\alpha_1)}\cdots B_{r(\alpha_j)s(\alpha_j)}}{A_{r(\alpha_1)s(\alpha_1)}\cdots A_{r(\alpha_j)s(\alpha_j)}}$
defined for each $\alpha=\alpha_1\alpha_2\dots \alpha_{|\alpha|}\in E^*$, $l\in \ZZ$, and $j\in \{1,\dots, |\alpha|\}$ we see that
$$M_l^i=\{\alpha\in iE^*\setminus \{i\}\ : \ K^{\alpha,l}_1,\dots, K^{\alpha,l}_{|\alpha|-1}\in \ZZ~\setminus~\{0\},\, K^\alpha_{|\alpha|}=0\}$$
is the set of all minimal strongly fixed paths for $l$ with range $i$  (see \cite[Lemma 18.4]{MR3581326}). Here $M^u_1=\{f,ef,eef,eeef,\dots\}$ is infinite, so $\Gg_{A,B}$ is non-Hausdorff by Proposition~\ref{when.hausdorff}.

Using Theorem~\ref{non.Hasu.EG} (or Theorem~\ref{thmc}{C}) proved later in this section we know that $\Oo_{A,B}^\text{alg}(R) \cong A_R(\Gg_{A,B})$. However, for this example the arguments simplify as follows:

Here the proof comes down to showing that for each
vertex ${w\in E^0}$ the indicator functions on the sets $\{\Theta_{(w,m,w)}: m\in \ZZ\}$ are linearly independent in $A_R(\Gg_{A,B})$, That is, for
each  finite subset $F$ of $\ZZ$,
\begin{equation*}
\sum_{m\in F} r_m1_{\Theta_{(w,m,w)}}=0 \ \ \Longrightarrow\ \  \text{each} \ r_m=0.
\end{equation*}
It turns out that a direct inspection of the sets $U_m\mydef \Theta_{(u,m,u)}$ and $V_m\mydef \Theta_{(v,m,v)}$ suffices. For each $m\in \ZZ$, we have $V_m\cap W=\emptyset$ for all $W\in \{U_n: n\in \ZZ\}\cup \{V_n: n\neq m\}$, and $[(u,m,u), eee\dots]\in U_m\setminus \big( \bigcup_{n\neq m} U_n \cup  \bigcup_{n\in \ZZ} V_n\big)$. So it is easy to see that the functions $1_{U_m}$ and $1_{V_m}$ are linearly independent. The infinite path $eee\dots$ with range $u$ is not strongly fixed, so this is an instance of case \eqref{caseONE}.
\end{example}

\begin{example}
\label{exp.2}
\begin{figure}
\centering
\[
\begin{tikzpicture}
    \node[circle, inner sep=1pt] (v2) at (5,0) {$u$};
    \node[circle, inner sep=1pt] (v3) at (7,1) {$v$};
            \node[circle, inner sep=1pt] (v3b) at (9,1) {$v'$};
        \node[circle, inner sep=1pt] (v3c) at (11,1) {$v''$};
    \node[circle, inner sep=1pt] (v4) at (9,-2) {$w$};
	     \draw[black,  -stealth, in=0, out=180] (v3c) to node[above, pos=0.5, black] {\small$0$} (v3b);
	     \draw[black,  -stealth, in=0, out=180] (v3b) to node[above, pos=0.5, black] {\small$0$} (v3);
	\path[->,every loop/.style={looseness=18}] (v3b)
	         edge[in=45,out=135,loop, black]  node[above, pos=0.5,black]{\small$1$} ();
	\path[->,every loop/.style={looseness=18}] (v4)
	         edge[in=45,out=135,loop, black]  node[above, pos=0.5,black]{\small$0$} ();
	         	\path[->,every loop/.style={looseness=18}] (v3c)
	         edge[in=45,out=135,loop, black]  node[above, pos=0.5,black]{\small$1$} ();
  	\draw[black,  -stealth, in=30, out=150] (v3) to node[above, pos=0.5, black] {\small$1$} (v2);
        \draw[black,  -stealth, in=20, out=160] (v3) to node[above, pos=0.5, black] {\small$$} (v2);
        \draw[black,  -stealth, in=330, out=210] (v4) to node[below, pos=0.5, black] {\small$1$} (v2);
        \draw[black,  -stealth, in=340, out=200] (v4) to node[below, pos=0.5, black] {\small$$} (v2);
        \draw[black,  -stealth, in=350, out=190] (v4) to node[below, pos=0.5, black] {\small$$} (v2);
\end{tikzpicture}
\]
\caption{Graph of a non-Hausdorff groupoid $\Gg_{A,B}$ without linear independence amongst $\{1_{\Theta_{(u,m,u)}}: m\in \ZZ\}$.} \label{fig.non.haus.not.read.2}
\end{figure}
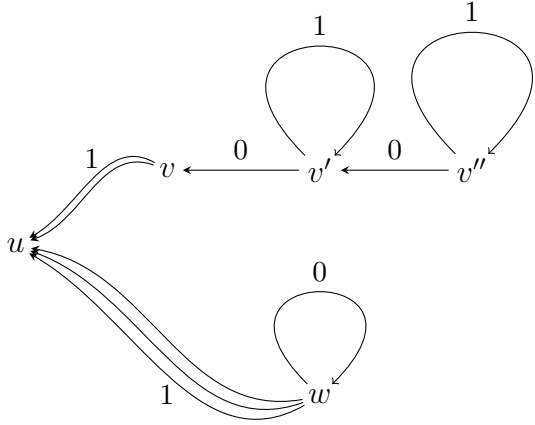
Let $\Gg_{A,B}$ be the groupoid of germs for the Katsura triple $(\ZZ,E,\varphi)$ associated to the matrices
$$A=
\left(\begin{smallmatrix}
0 & 2 & 0 & 0 & 3\\
0 & 0 & 1 & 0 & 0\\
0 & 0 & 1 & 1 & 0\\
0 & 0 & 0 & 1 & 0\\
0 & 0 & 0 & 0 & 1\end{smallmatrix} \right), \ \ \
B=
\left(\begin{smallmatrix}
0 & 1 & 0 & 0 & 1\\
0 & 0 & 0 & 0 & 0\\
0 & 0 & 1 & 0 & 0\\
0 & 0 & 0 & 1 & 0\\
0 & 0 & 0 & 0 & 0\end{smallmatrix} \right).$$The graph $E$ is illustrated on Figure~\ref{fig.non.haus.not.read.2} and the matrixes are given with respect to the ordering $u,v,v',v'',w$ of the vertices. Once again the groupoid $\Gg_{A,B}$ is non-Hausdorff (see Proposition~\ref{when.hausdorff}) but nevertheless we know that $\Oo_{A,B}^\text{alg}(R) \cong A_R(\Gg_{A,B})$.  Here it is the vertex $u$ that makes the arguments more challenging. All infinite paths with range $u$ are strongly fixed, cf.~case \eqref{caseTWO}. With $U_m\mydef\Theta_{(u,m,u)}$ for $m\in \ZZ$ we have
\begin{equation}
\label{non.lin.ind}
1_{U_0}+2\cdot 1_{U_1}+1_{U_2}=1_{U_3}+2\cdot 1_{U_4}+1_{U_5}.
\end{equation}
In particular the indicator functions $1_{U_0},\dots, 1_{U_5}$ in $A_R(\Gg_{A,B})$ are not linearly independent and the analysis used in Example~\ref{exp.1} is insufficient. Regardless, the lack of linear independence is compensated by a nice structure amongs the sets $U_m$. For $o\in \{v,w\}$,  set $U_m^o\mydef\{[(u,m,u), x]: x\in Z(u), x(1)=o\}$. Each $U_m$ admits a partition $U_m=U_m^v\sqcup U_m^w$ corresponding to minimal strongly fixed paths passing though $v$ and $w$ respectively. Any two of these sets $\{U_m^v, U_m^w :  m=0,\dots,5\}$ are either equal or disjoint(\footnote{
This is how we see \eqref{non.lin.ind}: we have ${\color{blue}U_0^{v}=U_2^{v}=U_4^{v}}$, ${\color{darkred}U_1^{v}=U_3^{v}=U_5^{v}}$,  ${\color{darkgreen}U_0^{w}=U_3^{w}}$, ${\color{violet}U_1^{w}=U_4^{w}}$, ${\color{olive}U_2^{w}=U_5^{w}}$, so $1_{U_0}+2\cdot 1_{U_1}+1_{U_2}=2\cdot1_{\color{blue}U_0^{v}}+2\cdot1_{\color{darkred}U_1^{v}}+1_{\color{darkgreen}U_0^{w}}+2\cdot 1_{\color{violet}U_1^{w}}+1_{\color{olive}U_2^{w}}=1_{U_3}+2\cdot 1_{U_4}+1_{U_5}$.
}). This mutual disjointness powers the proof of injectivity of $\pi_{\ZZ,E}$ as we shall see in the proof of Lemma~\ref{final110paths.new}.
\end{example}

We now return back to the proofs of Lemma~\ref{non-hausdorff} to Lemma~\ref{final110paths.new}.

\begin{lemma}
\label{non-hausdorff}
Let $(G,E,\varphi)$ be as in Notation~\ref{note}. Let $R$ be a unital commutative $^*$-ring. Then the $^*$-algebra $ A_R(\Gg_{\text{tight}}(\Ss_{G,E}))$  and the $^*$-homomorphism $\pi_{G,E}\colon \Ll_R(G,E) \to A_R(\Gg_{\text{tight}}(\Ss_{G,E}))$ of Proposition~\ref{nonzero} are $\ZZ$-graded.
\end{lemma}
\begin{proof}
Lemma~3.1 of \cite{MR3299719} generalises to not necessarily Hausdorff groupoids $\Gg\mydef\Gg_{\text{tight}}(\Ss_{G,E})$. Indeed, as in \cite{ClaExePar} the map $$c\colon \Gg \to \ZZ$$ given by $c([(\alpha,g,\beta), \mu])\mydef|\alpha|-|\beta|$ is well-defined and satisfies the cocycle identity $c(\gamma_1)c(\gamma_2)=c(\gamma_1\gamma_2)$ for any valid product $\gamma_1\gamma_2\in \Gg$. Now using that $A_R(\Gg)=\espan_R\{1_{\Theta_s} : s\in \Ss_{G,E}\}$ and the fact that $c$ is constant on each bisection $\Theta_s$ one can adopt the proof of  \cite[Lemma~3.1(1)]{MR3299719} as follows.

Fix any function $f=\sum_{s\in I} r_s1_{\Theta_s}\in A_R(\Gg)$ with $I\subseteq \Ss_{G,E}$ finite and each $r_s\neq 0$. For each $n\in \ZZ$ define
$$f_n\mydef\sum_{s\in I,\, c(s)=n} r_s1_{\Theta_s}.$$
 If follows that $f=\sum_n f_n$ is a sum of functions each in a graded component. Since the open supports $\supp(f_n)\mydef\{x\in \Gg : f_n(x)\neq 0\}$ of the functions $f_n$ are disjoint we deduce that  $f=0$ if and only if each $f_n=0$. It follows that the $^*$-algebra $ A_R(\Gg_{\text{tight}}(\Ss_{G,E}))$  and the $^*$-homomorphism $\pi_{G,E}\colon \Ll_R(G,E) \to A_R(\Gg_{\text{tight}}(\Ss_{G,E}))$ of Proposition~\ref{nonzero} are both $\ZZ$-graded.
\end{proof}

\begin{lemma}
\label{Um=Uniff}
Let $(G,E,\varphi)$ be as in Notation~\ref{note}. Let $R$ be a unital commutative $^*$-ring. Fix $u\in E^0$ and $n,m\in G$. Then
\begin{enumerate}
\item For each $\gamma\in E^*$ and $x\in E^\infty$ we have $x\in Z(\gamma)$ if and only if $m\act{}x\in Z(m\act{}\gamma)$.
\item If $y=[(u,m, m^{-1}\act{}u),  z]=[(u,n, n^{-1}\act{}u), z]$ for some $z\in E^\infty$, then $m\act{}z=n\act{}z$ and there exists $\alpha\in uE^*\setminus \{u\}$ that is strongly fixed by $nm^{-1}$ and satisfies $m\act{}z\in Z(\alpha)$.
\item Suppose that $\gamma\in uE^*\setminus \{u\}$ is strongly fixed by $nm^{-1}$. Then
$$[(u,m, m^{-1}\act{}u), z]=[(u,n, n^{-1}\act{}u), z] \ \text{for all} \ z\in Z(m^{-1}\act{}\gamma).$$
\item For any $v\in E^0\setminus \{u\}$ we have $\Theta_{(u,m,m^{-1}\act{}u)}\cap \Theta_{(v,n,n^{-1}\act{}v)}=\emptyset$.
\end{enumerate}
\end{lemma}
\begin{proof}
(1) Take any $\gamma\in E^*$ and $x\in Z(\gamma)$. By definition of $m\act{}x$ we have $(m\act{}x)(0,|\gamma|)=m\act{}(x(0,|\gamma|))=m\act{}\gamma$, so $m\act{}x\in Z(m\act{}\gamma)$. The converse is similar.

(2) Firstly we consider the case $m=n$. Then clearly $m\act{}z=n\act{}z$, and since $z\in Z(m^{-1}\act{}u)$ part (1) gives $m\act{}z\in Z(u)$. The element $\alpha\mydef (m\act{}z)_1\in  uE^*\setminus \{u\}$ is strongly fixed by $nm^{-1}$ and satisfies $m\act{}z\in Z(\alpha)$.

Suppose that $y=[(u,m, m^{-1}\act{}u), z]=[(u,n,n^{-1}\act{}u), z]$ for some $z\in E^\infty$ and $m\neq n$. With $s\mydef(u,m, m^{-1}\act{}u)$ and $t\mydef(u,n, n^{-1}\act{}u)$ we know that $[s,z]=[t,z]$ so there exists $e\mydef(\beta,0,\beta)$ with $\beta\in E^*$ such that $se=te$ and $z\in Z(\beta)$. With $\beta_1\mydef m\act{}\beta$ and $\beta_2\mydef n\act{}\beta$ we get
\begin{align*}
se&=(\beta_1,\varphi(m^{-1},\beta_1)^{-1}, m^{-1}\act{}\beta_1), \ \text{and} \ te=(\beta_2,\varphi(n^{-1},\beta_2)^{-1}, n^{-1}\act{}\beta_2),
\end{align*}
so $\beta_1=\beta_2$ and $\varphi(m^{-1},\beta_1)=\varphi(n^{-1},\beta_2)$. Let $\alpha\mydef \beta_1$. Then $n^{-1}\act{}\alpha=m^{-1}\act{}\alpha$ and $\varphi(n^{-1},\alpha)=\varphi(m^{-1},\alpha)$, so $\alpha$ is strongly fixed by $nm^{-1}$ \cite[Proposition~5.6]{MR3581326}. Since $z\in Z(\beta)$ it follows that $m\act{}z \in Z(m\act{}\beta)= Z(\alpha)$. Now since $z\in Z(m^{-1}\act{}u)$ we get $m\act{}z\in Z(u)$ and $r(\alpha)=u$. The element $\alpha\in uE^*$ can not be a vertex, because that would imply $m=\varphi(m^{-1},\alpha)^{-1}=\varphi(n^{-1},\alpha)^{-1}=n$. So $\alpha\in uE^*\setminus \{u\}$. Now we prove that $m\act{}z=n\act{}z$. Using that $\varphi(g^{-1},\gamma)=\varphi(g, g^{-1}\act{}\gamma)^{-1}$, see \cite[Proposition~2.6]{MR3581326}, we get $\varphi(m,\beta)=\varphi(m^{-1}, (m^{-1})^{-1}\act{}\beta)^{-1}=\dots=\varphi(n,\beta)$. Also $m\act{}\beta=n\act{}\beta$. Since $z\in Z(\beta)$, $m\act{}z=m\act{}(\beta z')=(m\act{}\beta)(\varphi(m,\beta)\act{}z')=\dots=n\act{}z$ for suitable $z'\in E^\infty$.

(3) Suppose that $\gamma\in uE^*\setminus \{u\}$ is strongly fixed by $nm^{-1}$. Take any $z\in Z(m^{-1}\act{}\gamma)$. Set $s\mydef(u,m, m^{-1}\act{}u)$ and $t\mydef(u,n, n^{-1}\act{}u)$.
Since $r(\gamma)=u$ and since $\gamma$ is strongly fixed by $nm^{-1}$, we have $z\in Z(m^{-1}\act{}u)=Z(n^{-1}\act{}u)$. Hence both $[s,z]$ and $[t,z]$ make sense. We claim that $[s,z]=[t,z]$. To see this, set $e\mydef(m^{-1}\act{}\gamma,0,m^{-1}\act{}\gamma)$. As $\gamma$ is strongly fixed by $nm^{-1}$, we have $e=(n^{-1}\act{}\gamma,0,n^{-1}\act{}\gamma)$, and hence
$$se=(\gamma,\varphi(m^{-1},\gamma)^{-1}, m^{-1}\act{}\gamma)=(\gamma,\varphi(n^{-1},\gamma)^{-1}, n^{-1}\act{}\gamma)=te.$$
Since $z\in Z(m^{-1}\act{}\gamma)=Z(n^{-1}\act{}\gamma)$ it follows that $[s,z]=[t,z]$.

(4) Elements of $\Theta_{(u,m,m^{-1}\act{}u)}$ have range in $Z(u)$ (\cite{MR2419901}), so $\Theta_{(u,m,m^{-1}\act{}u)}\cap \Theta_{(v,n,n^{-1}\act{}v)}\neq \emptyset$ implies $Z(u)\cap Z(v)\neq \emptyset$ and $u=v$.
\end{proof}

\begin{lemma}
\label{with111paths}
Let $(G,E,\varphi)$ be as in Notation~\ref{note}. Let $R$ be a unital commutative $^*$-ring. Fix $u\in E^0$. Suppose $x\in Z(u)$ is not strongly fixed. Then the indicator functions on the sets $\{\Theta_{(u,m,m^{-1}\act{}u)}, m\in G\}$ are linearly independent, i.e., for
each finite subset $F$ of $G$,
\begin{equation*}
\sum_{m\in F} r_m1_{\Theta_{(u,m,m^{-1}\act{}u)}}=0 \ \ \Longrightarrow\ \  r_m=0 \ \ \text{ for each} \ \  m\in F.
\end{equation*}
\end{lemma}
\begin{proof}
Set $h\mydef \sum_{m\in F} r_m1_{\Theta_{(u,m,m^{-1}\act{}u)}}$ and for each $n\in G$, set $x^{(n)}\mydef [(u,n,n^{-1}\act{}u),n^{-1}\act{}x]$. Suppose that $x^{(n)}\in \Theta_{(u,m,m^{-1}\act{}u)}$ for some $n, m\in G$. It follows that $[(u,m, m^{-1}\act{}u), z]=[(u,n, n^{-1}\act{}u), z]$ for $z=n^{-1}\act{}x\in E^\infty$ . By Lemma~\ref{Um=Uniff} we have $x= m\act{}z=n\act{}z$ and there exists $\alpha\in uE^*$ that is strongly fixed by $nm^{-1}$ and satisfies $x\in Z(\alpha)$. But $x$ is not strongly fixed, so $n=m$. Hence $x^{(n)}\notin \Theta_{(u,m,m^{-1}\act{}u)}$ for $m\neq n$, so $h(x^{(m)})=r_m$.
\end{proof}

\begin{lemma}
\label{with110paths.new}
Let $(G,E,\varphi)$ be as in Notation~\ref{note}. Let $R$ be a unital commutative $^*$-ring. Fix $u\in E^0$. For $(m,\gamma)\in G\times uE^*$ define
$$U_{(m,\gamma)}\mydef \{[(u,m,m^{-1}\act{}u), x]: x\in Z(m^{-1}\act{}\gamma)\}.$$
Let $\mathcal{F}$ be the set of all minimal strongly fixed paths with range $u$ (see Definition~\ref{strongly.fixed}).
Suppose that $Z(\gamma)\cap Z(\gamma')=\emptyset$ whenever $\gamma\neq\gamma'\in \mathcal{F}$. Set $\mathcal{P}\mydef \{U_{(m,\gamma)} : (m,\gamma)\in G\times \mathcal{F}\}$. Then
\begin{enumerate}
\item \label{1.lem}The sets in $\mathcal{P}$ are compact open bisections.
\item  \label{2.lem} For $U\mydef U_{(m,\gamma)}$ and $V\mydef U_{(n,\eta)}\in\mathcal{P}$, the following are equivalent:
\begin{enumerate}
\item\label{2.a} $U\cap V\neq\emptyset$;
\item\label{2.b}  $\gamma=\eta \ \text{and} \ \gamma \ \text{is strongly fixed by}\ nm^{-1}$; and
\item\label{2.c}   $U=V$.
\end{enumerate}
\item For any $V_1,V_2\in \mathcal{P}$ either $V_1=V_2$ or $V_1\cap V_2=\emptyset$.
\item For any $V_i=U_{(m,\gamma_i)}\in \mathcal{P}$ we have $V_1\cap V_2=\emptyset \Leftrightarrow \gamma_1\neq\gamma_2$.
\item If $\mathcal{F}$ is finite and every $x\in Z(u)$ is strongly fixed, then
$$1_{\Theta_{(u,m,m^{-1}\act{}u)}}=\sum_{\gamma\in \mathcal{F}} 1_{U_{(m,\gamma)}}, \ \ \text{for each}\ m\in G.$$
\item For each $U\mydef U_{(m,\gamma)}\in \mathcal{P}$ set $p_U\mydef s_{\gamma,\varphi(m^{-1},\gamma)^{-1}}s_{m^{-1}\act{}\gamma,e_G}^*$.
With  $\pi_{G,E}$ as in Proposition~\ref{nonzero},
$$\pi_{G,E}(p_U)=1_{\Theta_{(\gamma,\varphi(m^{-1},\gamma)^{-1}, m^{-1}\act{}\gamma)}}=1_U.$$
\item For $U\mydef U_{(m,\gamma)}$ and $V\mydef U_{(n,\eta)}\in\mathcal{P}$, the following are equivalent:
\begin{enumerate}
\item[(c)]  $U=V$; and
\item[(d)]  $p_U=p_V$.
\end{enumerate}
\item If $\mathcal{F}$ is finite and every $x\in Z(u)$ is strongly fixed, then
$$p_{u,m}=\sum_{\gamma\in \mathcal{F}} p_{U_{(m,\gamma)}} \ \ \text{for each}\ m\in G.$$
\end{enumerate}
\end{lemma}
\begin{proof}
 \eqref{1.lem} Notice that $Z(m^{-1}\act{}\gamma)$ is an open subset of $Z(m^{-1}\act{}u)$. The result now follows from \cite[ Proposition~4.18]{MR2419901}.

\eqref{2.a}$\Rightarrow$\eqref{2.b} Suppose that $y\in U \cap V$. Then $y=[(u,m, m^{-1}\act{}u), z]=[(u,n,u; n^{-1}\act{}u), z]$ for some $z\in E^\infty$. By Lemma~\ref{Um=Uniff} we have $x\mydef m\act{}z=n\act{}z$ and there exists $\alpha\in uE^*$ that is strongly fixed by $nm^{-1}$ with $x\in Z(\alpha)$. Since $z\in Z(m^{-1}\act{}\gamma)$, Lemma~\ref{Um=Uniff} gives $x=m\act{}z\in Z(m\act{}(m^{-1}\act{}\gamma))=Z(\gamma)$. Similarly $x\in Z(\eta)$, so $\gamma=\eta$. We may assume $n\neq m$. Since $x\in Z(\alpha)\cap Z(\gamma)$ we deduce that $\gamma=\alpha$ is strongly fixed by $nm^{-1}$.

\eqref{2.b}$\Rightarrow$\eqref{2.c} Now suppose that  that $\gamma=\eta$ is strongly fixed by $nm^{-1}$. Take any $y\in U$, say $y=[(u,m, m^{-1}\act{}u), z']$ for some $z'\in Z(m^{-1}\act{}\gamma)$. Using $\gamma$ is strongly fixed by $nm^{-1}$ it follows from Lemma~\ref{Um=Uniff} that
$$[(u,m, m^{-1}\act{}u), z]=[(u,n, n^{-1}\act{}u), z] \ \text{for all} \ z\in Z(m^{-1}\act{}\gamma).$$
Since $Z(m^{-1}\act{}\gamma)=Z(n^{-1}\act{}\gamma)$ we get $y\in V$. By symmetry $U=V$.

\eqref{2.c}$\Rightarrow$\eqref{2.a} is trivial, completing the proof of (2). Both (3) and (4) follow from \eqref{2.lem}.

(5) Recall that $\Theta_{(u,m,m^{-1}\act{}u)}=\{[(u,m,m^{-1}\act{}u), x]: x\in Z(m^{-1}\act{}u)\}$. Take any element $y\mydef [(u,m,m^{-1}\act{}u), z]\in \Theta_{(u,m,m^{-1}\act{}u)}$, so $z\in Z(m^{-1}\act{}u)$. Then $x\mydef m\act{}z\in Z(u)$. Since $x\in Z(u)$ is strongly fixed, there exists some initial segment $\gamma\in  uE^*\setminus\{u\}$ of $x$ such that $\gamma$ is minimal strongly fixed. Clearly $\gamma\in \mathcal{F}$. Moreover, $z=m^{-1}\act{}x\in Z(m^{-1}\act{}\gamma)$, so $y\in U_{(m,\gamma)}$. Conversely, for each $\gamma \in \mathcal{F}$ we have $Z(m^{-1}\act{}\gamma)\subseteq Z(m^{-1}\act{}u)$, so
$$\Theta_{(u,m,m^{-1}\act{}u)}=\bigcup_{\gamma\in \mathcal{F}} U_{(m,\gamma)}.$$
So (4) implies that the $U_{(m,\gamma)}$ are mutually disjoint giving the desired result.

(6) By direct computation (cf.~Remark~\ref{rem.pi.map}) one can verify that $\pi_{G,E}(p_U)=1_{\Theta_{(\gamma,\varphi(m^{-1},\gamma)^{-1}, m^{-1}\act{}\gamma)}}$. With $s\mydef(u,m,m^{-1}\act{}u)$ and $e\mydef(m^{-1}\act{}\gamma,0,m^{-1}\act{}\gamma)$ we have $[s, x]=[se,x]$ for each $x\in Z(m^{-1}\act{}\gamma)$. It follows that
\begin{align*}
U_{(m,\gamma)}&=\{[s,x]: x\in Z(m^{-1}\act{}\gamma)\}=\{[se,x]: x\in Z(m^{-1}\act{}\gamma)\}\\
&=\{[(\gamma,\varphi(m^{-1},\gamma)^{-1}, m^{-1}\act{}\gamma), x]: x\in Z(m^{-1}\act{}\gamma)\}\\
&=\Theta_{(\gamma,\varphi(m^{-1},\gamma)^{-1}, m^{-1}\act{}\gamma)}.
\end{align*}

(7) If $p_U=p_V$ then $U=V$ by (6). Conversely if $U=V$ then part (2) gives that $\gamma=\eta$ and $\gamma$ is strongly fixed by $nm^{-1}$. Hence $m^{-1}\act{}\gamma=n^{-1}\act{}\gamma$ and $\varphi(m^{-1},\gamma)=\varphi(n^{-1},\gamma)$, see \cite[Proposition~5.6]{MR3581326}. So $s_{\gamma,\varphi(m^{-1},\gamma)^{-1}}s_{m^{-1}\act{}\gamma,e_G}^*=s_{\eta,\varphi(n^{-1},\eta)^{-1}}s_{n^{-1}\act{}\eta,e_G}^*$ and $p_U=p_V$.

(8) Fix $m\in G$. Define $J\mydef{\{(\gamma,\varphi(m^{-1},\gamma)^{-1}, m^{-1}\act{}\gamma): \gamma\in \mathcal{F}\}}$ and $I\mydef\{(u,m,m^{-1}\act{}u)\}$. Then (5) and (6) give
$$ \bigsqcup_{s\in J}\Theta_{s}=\bigsqcup_{\gamma\in \mathcal{F}}\Theta_{(\gamma,\varphi(m^{-1},\gamma)^{-1}, m^{-1}\act{}\gamma)} =\bigsqcup_{\gamma\in \mathcal{F}}U_{(m,\gamma)}={\Theta_{(u,m,m^{-1}\act{}u)}}= \bigsqcup_{s\in I}\Theta_{s}.$$
With $\Jj$ as in Lemma~\ref{lem.tU} we have $I,J\in \Jj$. Using Lemma~\ref{lem.tU} we get
$$ \sum_{(\alpha,g,\beta)\in I} s_{\alpha,g}s_{\beta,e_G}^*=\sum_{(\alpha,g,\beta)\in J} s_{\alpha,g}s_{\beta,e_G}^*.$$
It follows that
\begin{align*}
p_{u,m}&=\sum_{(\alpha,g,\beta)\in J} s_{\alpha,g}s_{\beta,e_G}^*=\sum_{\gamma\in \mathcal{F}} p_{U_{(m,\gamma)}}.
\end{align*}
as claimed.
\end{proof}

\begin{lemma}
\label{final110paths.new}
Let $(G,E,\varphi)$ be as in Notation~\ref{note}. Let $R$ be a unital commutative $^*$-ring. Fix $u\in E^0$. Suppose that every $x\in Z(u)$ is strongly fixed. Let $\mathcal{F}$ be the set of all minimal strongly fixed paths with range $u$. Suppose that  $\mathcal{F}$ is finite and $Z(\gamma)\cap Z(\gamma')=\emptyset$ whenever $\gamma\neq \gamma'$ in $\mathcal{F}$. Then for each finite subset $F$ of $G$,
\begin{equation*}
\sum_{m\in F} r_m1_{\Theta_{(u,m,m^{-1}\act{}u)}}=0 \ \ \Longrightarrow\ \  \sum_{m\in F} r_mp_{m,u}=0.
\end{equation*}
\end{lemma}
\begin{proof}
Suppose that  $h\mydef \sum_{m\in F} r_m1_{\Theta_{(u,m,m^{-1}\act{}u)}}$ is the zero function. For each $(m,\gamma)\in F\times \mathcal{F}$ define $r_{(m,\gamma)}\mydef r_m$. Using Lemma~\ref{with110paths.new}
$$h=\sum_{m\in F} r_m\sum_{\gamma\in \mathcal{F}} 1_{U_{(m,\gamma)}}=\sum_{(m,\gamma)\in F\times \mathcal{F}}r_m1_{U_{(m,\gamma)}}=\sum_{p\in F\times \mathcal{F}}r_p1_{U_{p}}.$$
For each $p\in F\times \mathcal{F}$ set $I(p)\mydef\{p'\in F\times \mathcal{F} :  U_{p}=U_{p'}\}$. Since $F\times \mathcal{F}$ is finite there exist a smallest set $P_{\min}\subseteq F\times \mathcal{F}$ such that $U_{p}\neq U_{p'}$ for $p\neq p'\in P_{\min}$ and $F\times \mathcal{F}=\bigcup_{p\in P_{\min}} I(p)$. For each $p\in P_{\min}$ set $s_p=\sum_{p'\in I(p)}r_{p'}$. Then
$$h=\sum_{p\in P_{\min}}\sum_{p'\in I(p)}r_{p'}1_{U_{p}}=\sum_{p\in P_{\min}}s_p1_{U_{p}}.$$
If $p_1\neq p_2\in P_{\min}$ then $U_{p_1}\neq U_{p_2}\in \mathcal{P}$, so $U_{p_1}\cap U_{p_2} =\emptyset$ by Lemma~\ref{with110paths.new}. So for $p\in P_{\min}$, we have $h(\gamma)= s_p$ for all $\gamma\in U_p$. Thus $s_p=0$ for each $p\in P_{\min}$. Hence
$$0=\sum_{p\in P_{\min}}s_pp_{U_{p}}=\sum_{p\in P_{\min}}\sum_{p'\in I(p)}r_{p'}p_{U_{p}}.$$
For each $p'\in I(p)$ we have $U_{p}=U_{p'}\in \mathcal{P}$, so Lemma~\ref{with110paths.new} gives $p_{U_{p}}=p_{U_{p'}}$. Thus
$$0=\sum_{p\in P_{\min}}\sum_{p'\in I(p)}r_{p'}p_{U_{p}}=\sum_{p\in P_{\min}}\sum_{p'\in I(p)}r_{p'}p_{U_{p'}}=\sum_{p'\in F\times \mathcal{F}}r_{p'}p_{U_{p'}}.$$
Lemma~\ref{with110paths.new} gives
$$\sum_{m\in F} r_mp_{m,u}=\sum_{m\in F} r_m\sum_{\gamma\in \mathcal{F}} p_{U_{(m,\gamma)}}
=\sum_{(m,\gamma)\in F\times \mathcal{F}}r_mp_{U_{(m,\gamma)}}=0$$
as claimed.
\end{proof}

We are now able to prove Theorem~\ref{non.Hasu.EG}.

\begin{proof}[Proof of Theorem~\ref{non.Hasu.EG}]
By the Graded Uniqueness Theorem~\ref{thma}{A}, $\pi_{G,E}$ is injective if and only if its restriction to $\Dd$ is injective, that is, if and only if for each finite subset $P$ of $E^0\times  G$,
\begin{equation*}
\sum_{(u,m)\in P} r_{(u,m)}1_{\Theta_{(u,m,m^{-1}\act{}u)}}=0 \ \ \Longrightarrow\ \ \sum_{(u,m)\in P} r_{(u,m)}p_{u,m}=0.
\end{equation*}
By the last item of Lemma~\ref{Um=Uniff} it suffices to consider one vertex at a time and show that: for each $u\in E^0$ and each finite subset $F$ of $G$,
\begin{equation*}
\label{final.thm.lin}
\sum_{m\in F} r_m1_{\Theta_{(u,m,m^{-1}\act{}u)}}=0 \ \ \Longrightarrow\ \  \sum_{m\in F} r_mp_{m,u}=0.
\end{equation*}
So suppose that $h\mydef \sum_{m\in F} r_m1_{\Theta_{(u,m,m^{-1}\act{}u)}}=0$.

Firstly suppose that there exists  $x\in Z(u)$ that is not strongly fixed. By Lemma~\ref{with111paths} the indicator functions of the sets $\{\Theta_{(u,m,m^{-1}\act{}u)}, m\in F\}$ are linearly independent. Thus each $r_m=0$, and so $\sum_{m\in F} r_mp_{u,m}=0$.

Secondly suppose that $\mathcal{F}_u$, the set of all minimal strongly fixed paths with range $u$, is finite. Then Lemma~\ref{final110paths.new} gives $\sum_{m\in F} r_mp_{u,m}=0$.
\end{proof}

Having Theorem~\ref{non.Hasu.EG} at our disposal we can now prove Theorem~\ref{thmc}{C} by simply verifying that each Katsura triple $(\ZZ,E,\varphi)$ with $E$ finite satisfies the conditions set out in Theorem~\ref{non.Hasu.EG}. To do this we recall some terminology. Let $E$ be any directed graph and let $A,B$ be integer valued $E^0 \times E^0$ matrices. Recall that $E^*$ denotes the set of finite paths in $E$. For a path $\alpha\in E^*$ and $i\in\{1,\dots, |\alpha|\}$ we let $\alpha_i$ be the $i$th edge of $\alpha$ so $\alpha=\alpha_1\alpha_2\dots \alpha_{|\alpha|}$. For $l\in \NN$ and $i\in\{1,\dots, |\alpha|\}$ define $K^{\alpha,l}_i\mydef l\,\frac{B_{r(\alpha_1)s(\alpha_1)}\cdots B_{r(\alpha_i)s(\alpha_i)}}{A_{r(\alpha_1)s(\alpha_1)}\cdots A_{r(\alpha_i)s(\alpha_i)}}$, cf.~\cite{ExePar, MR3581326}.
We finally proceed with the proof of Theorem~\ref{thmc}{C}.

\begin{proof}[Proof of Theorem~\ref{thmc}{C}]
Fix any vertex $u\in E^0$. We must show that $Z(\gamma)\cap Z(\gamma')=\emptyset$ whenever $\gamma\neq\gamma'\in \mathcal{F}_u$ and that $\mathcal{F}_u$ is finite whenever every $x\in Z(u)$ is strongly fixed and $N = |E^0|<\infty$.

Fix $\gamma\neq\gamma'$ of $\mathcal{F}_u$. Since $\gamma$ is minimal strongly fixed, it is  minimal strongly fixed by some $l\geq 1$. Hence
\begin{align*}
&\varphi(l, \gamma_1\dots \gamma_{i})=K^{\gamma,l}_{i}\in \ZZ\smallspace\setminus\smallspace \{0\}  \quad\text{for} \ i<|\gamma|, \ \text{and}\\
&\varphi(l, \gamma)=K^{\gamma,l}_{|\gamma|}=0.
\end{align*}

It follows that $B_{r(\gamma_i)s(\gamma_i)}\neq 0$ for all $i<|\gamma|$ and $B_{r(\gamma_i)s(\gamma_i)}= 0$ for  $i=|\gamma|$. By symmetry, if one of $\gamma, \gamma'$ is an initial segment of the other then they must have the same length. Hence $Z(\gamma)\cap Z(\gamma')\neq \emptyset$ implies $\gamma=\gamma'$, so $Z(\gamma)\cap Z(\gamma')=\emptyset$ whenever $\gamma\neq\gamma'$.

Now suppose that every $x\in Z(u)$ is strongly fixed and $N = |E^0|<\infty$. Fix $\beta\in uE^{N}$. We claim that $\beta$ is strongly fixed. To see this let $x\in E^\infty$ be an infinite path having $\beta$ as an initial segment. Find the shortest initial segment $\beta_x\in E^*\setminus \{u\}$ of $x$ that is strongly fixed. Say $\beta_x$ is fixed by $m\neq 0$. We suppose that $|\beta_x|> N$ and derive a contradiction. Since $|\beta_x|> N$, any initial segment $\beta_n$ of $x$ of length $n\in \{1,\dots, N\}$ must satisfy $\varphi(m,\beta_n)\neq 0$ because $\beta_x$ is the shortest segment that is strongly fixed by $m$. Since $N=|E^0|$, one of the paths $\beta_n$ has the form $\alpha\gamma$ where $\gamma$ is a loop. By construction $\varphi(m, \alpha\gamma)\neq 0$ so $B_{r(\alpha_i)s(\alpha_i)}\neq 0$ for $i\leq |\alpha|$ and $B_{r(\gamma_i)s(\gamma_i)}\neq 0$ for all $i\leq |\gamma|$. Hence $x\mydef \alpha\gamma\gamma\gamma\dots\in Z(u)$ is not strongly fixed. This contradicts that every $x\in Z(u)$ is strongly fixed. We conclude that $|\beta_x|\leq N$.
Since $|\beta|=N$, $\beta_x$ as an initial segment
of $\beta$. So $\beta$ is strongly fixed. Since $uE^{N}$ is finite and $\mathcal{F}_u$ is a subset of $uE^{N}$ we deduce that $\mathcal{F}_u$ is finite. The result now follows from Theorem~\ref{non.Hasu.EG}.
\end{proof}

\begin{remark}
It may happen that all the sets $\{\Theta_{(w,m,w)}: m\in \ZZ\}$ are identical. This is the case, for example, for $w\in E^0$ in Example~\ref{exp.2}. In this situation, the corresponding row of $B$ is identically $0$.
\end{remark}

\bibliographystyle{plain}

\end{document}